\documentclass[letterpaper]{amsart}
\usepackage{mathrsfs}  
\usepackage[square]{natbib}  
\usepackage{amssymb}  
\usepackage{tikz}  
\usetikzlibrary{shapes,arrows}

\swapnumbers  
\newtheorem{clm}{Claim}[section]
\newtheorem{main}[clm]{Main Claim}

\theoremstyle{definition}
\newtheorem{dfn}[clm]{Definition}

\theoremstyle{remark}
\newtheorem{hyp}[clm]{Hypothesis}
\swapnumbers
\newtheorem{case}{Case}

\DeclareMathOperator{\dom}{dom}  
\DeclareMathOperator{\val}{val}  
\DeclareMathOperator{\lcm}{lcm}  
\DeclareMathOperator{\conv}{conv}  
\newcommand{\abs}[1]{\left\lvert#1\right\rvert}  

\newcommand{\N}{\mathbb{N}}
\newcommand{\Q}{\mathbb{Q}}
\newcommand{\0}{\emptyset}
\newcommand{\ntnt}{\ensuremath{(0, 0)}}
\newcommand{\cv}[2]{\conv\left(V_{#1}(#2)\right)}
\newcommand{\cvn}{\cv{-1}{c}}  
\newcommand{\cvt}{\cv{0}{c}}   
\newcommand{\cvw}{\cv{1}{c}}   
\newcommand{\cvl}{\cv{\ell}{c}}  

\newcommand{\n}{\mathbf{n}}
\newcommand{\prt}{\bar{t}}
\newcommand{\prs}{\bar{s}}
\newcommand{\pru}{\bar{u}}
\newcommand{\C}{\mathfrak{C}}
\newcommand{\subC}{\mathscr{C}}
\newcommand{\subB}{\mathscr{B}}
\newcommand{\subP}{\mathscr{P}}
\newcommand{\perX}{{\operatorname{Per}(X)}}
\DeclareMathOperator{\maj}{maj}
\DeclareMathOperator{\majcl}{maj-cl}
\DeclareMathOperator{\pr}{pr}
\newcommand{\prbl}{\pr^\textnormal{bl}}
\DeclareMathOperator{\prcl}{pr-cl}
\DeclareMathOperator{\tor}{Tor}

\newcounter{dlcount}

\begin{document}

\title[Majority Decisions]{What majority decisions are possible with possible abstaining}

\author[P. Larson]{Paul Larson}
\author[N. Matteo]{Nick Matteo}
\address[Paul Larson, Nick Matteo]{Department of Mathematics\\Miami University\\
Oxford, OH 45056\\USA}
\email{larsonpb@muohio.edu, matteona@muohio.edu}
\thanks{The research of the first author is supported by NSF grant DMS-0801009.}
\author[S. Shelah]{Saharon Shelah}
\address[Saharon Shelah]{The Hebrew University of Jerusalem\\Einstein Institute of Mathematics\\
Edmond J. Safra Campus, Givat Ram\\Jerusalem 91904\\Israel}
\address[Saharon Shelah]{Department of Mathematics\\Hill Center-Busch Campus\\
Rutgers, The State University of New Jersey\\
110 Frelinghuysen Road\\Piscataway, NJ 08854-8019\\USA}
\thanks{Partially supported by the United States-Israel Binational Science Foundation.}

\date{\today}
\keywords{choice function, majority decision, Condorcet's paradox, tournament.}
\subjclass[2000]{Primary 91B14; Secondary 05C20}

\begin{abstract}
  Suppose we are given a family of choice functions on pairs from a given
  finite set.  The set is considered as a set of alternatives (say candidates for
  an office) and the functions as potential ``voters.''  The question is, what
  choice functions agree, on every pair, with the majority of some finite
  subfamily of the voters?  For the problem as stated, a complete
  characterization was given in \citet{shelah2009mdp}, but here we allow each
  voter to abstain.  There are four cases.
\end{abstract}
\maketitle

\section{Introduction}
Condorcet's ``paradox'' demonstrates that given three candidates A, B, and C,
majority rule may result in the society preferring A to B, B to C, and C to
A \citep{condorcet}.  \citet{McG53} proved a far-reaching extension of
Condorcet's paradox: For every asymmetric relation $R$ on a set $X$ of $\n$
candidates, there are $m$ linear order relations on $X$: $R_1, R_2, \ldots,
R_m$, with $R$ as their strict simple majority relation.  I.E. for every $a, b
\in X$,
\[
a \mathrel{R} b \iff \abs{\{i : a \mathrel{R_i} b \}} > \frac{m}{2}.
\]
In other words, given any set of choices from pairs from a set of $\n$
candidates, there is a population of $m$ voters, all with simple linear-order
preferences among the candidates, who will yield the given outcome for each pair
in a majority-rule election between them.

McGarvey's proof gave $m = \n(\n - 1)$.  \Citet{Ste59} found a construction
with $m = \n$ and noticed that a simple counting argument implies that $m$ must
be at least $\frac{\n}{\log \n}$.  \Citet{ErMo64} were able to give a
construction with $m = O\left(\frac{\n}{\log \n}\right)$.  \Citet{Alo02} showed
that there is a constant $c_1 > 0$ such that, given any asymmetric relation $R$,
there is some $m$ and linear orders $R_1, \ldots, R_m$ with
\[
a \mathrel{R} b \iff \abs{\{i : a \mathrel{R_i} b \}} > \frac{m}{2} + c_1 \sqrt{m}
\]
and that this is not true for any $c_2 > c_1$.

Gil Kalai asked to what extent the assertion of McGarvey's theorem holds if we
replace the linear orders by an arbitrary isomorphism class of choice functions
on pairs of elements.  Namely, when can we guarantee that every asymmetric
relation $R$ on $X$ could result from a finite population of voters, each using
a given kind of asymmetric relation on $X$?

Of course we must define what we mean by ``kind'' of asymmetric relation.
Let $\binom{X}{k}$ denote the family of subsets of $X$ with $k$ elements:
\[
\binom{X}{k} = \{Y \subseteq X : \abs{Y} = k\}.
\]
Either an individual voter's preferences among
candidates, or the outcomes which would result from each two-candidate
election, may be represented as
\begin{itemize}
  \item
    An asymmetric relation $R$ where $a \mathrel{R} b$ iff $a$ beats $b$;
    it is possible that $a \mathrel{\not \hspace{-3pt} R} b$,
    $b \mathrel{\not\hspace{-3pt}R} a$, and $a \neq b$.
  \item
    A choice function defined on some subfamily of $\binom{X}{2}$, choosing the
    winner in each pair.  Such a choice function is called ``full'' if its domain
    is all of $\binom{X}{2}$, and ``partial'' otherwise.
  \item
    An oriented graph, i.e.\ a directed graph with nodes $X$ and edges $a \to b$
    when $a$ beats $b$.
\end{itemize}
We shall treat these representations as largely interchangeable throughout this
paper.  Total asymmetric relations, full choice functions, and tournaments
(complete oriented graphs) all correspond to the case of no abstaining.
$\tor(c)$ will denote the oriented graph associated with a choice function $c$.
For any set $X$, $c \mapsto \tor(c)$ is a bijection of full choice functions
onto tournaments on $X$, and a bijection of all choice functions onto oriented
graphs on $X$.

\begin{hyp}
  Assume
  \begin{list}{(\alph{dlcount})}{\usecounter{dlcount}}
    \item
      $X$ is a finite set with $\n = \abs{X}$ and $\n \geq 3$.
    \item
      $\C$ is the set of choice functions on pairs in $X$;
      \[
      \C = \left\{ c \colon Y \to X : Y \subseteq \binom{X}{2},
      \forall \{x,y\} \in Y \, c\{x,y\} \in \{x,y\} \right\}.
      \]
    When $c\{x,y\}$ is not defined it is interpreted as abstention or having no
    preference.
  \end{list}
\end{hyp}

\begin{dfn}
  \begin{list}{(\alph{dlcount})}{\usecounter{dlcount}}
    \item
      $\perX$ is the set of permutations of $X$.
    \item
      Choice functions $c$ and $d$ are \emph{symmetric} iff there is
      $\sigma \in \perX$ such that
      \[
      d\{\sigma(x), \sigma(y)\} = \sigma(x) \iff c\{x, y\} = x;
      \]
      we write $d = c^\sigma$.
    \item
      A set of choice functions $\subC \subseteq \C$ is \emph{symmetric} iff it is
      closed under permutations of $X$.
      So for each $\sigma \in \perX$, if $c \in \subC$ then $c^\sigma \in \subC$.
  \end{list}
\end{dfn}
Choice functions $c$ and $d$ are symmetric iff $\tor(c)$ and $\tor(d)$ are
isomorphic graphs.
Note that symmetry of choice functions is an equivalence relation.
These symmetric sets of choice functions are in fact what was meant by ``kind
of asymmetric relation.''

The main result of \citet{shelah2009mdp} pertained to full choice functions for
the voters.  It was shown that an arbitrary choice function $d$ could result
from a symmetric set $\subC$, i.e.\ for each $d$ there is a finite set
$\{c_1, \ldots, c_m\} \subseteq \subC$ such that
\[
d\{x, y\} = x \iff \abs{\{i : c_i\{x,y\} = x\}} > \frac{m}{2},
\]
iff for some $c \in \subC$ and $x \in X$,
\[
\abs{\{y : c\{x,y\} = y\}} \neq \frac{\n - 1}{2}.
\]
We shall call this condition ``imbalance.''
\begin{dfn}
  For a choice function $c$,
  \begin{list}{(\alph{dlcount})}{\usecounter{dlcount}}
    \item
      for any pair $(x, y)$ in $X^2$, the \emph{weight} of $x$ over $y$ for $c$ is
      \[
      W^x_y(c) =
      \begin{cases}
        \phantom{-}1 & \text{if } c\{x,y\} = x \\
        \phantom{-}0 & \text{if } \{x,y\} \notin \dom c \\
        -1 & \text{if } c\{x,y\} = y
      \end{cases}
      \]
    \item
      $c$ is \emph{balanced} iff
      \[
      \forall x \in X, \sum_{y \in X} W^x_y(c) = 0.
      \]
      That is, $\abs{\{y : c\{x, y\} = x\}} = \abs{\{y : c\{x,y\} = y\}}$,
      for every $x$ in $X$.
    \item
      $c$ is \emph{imbalanced} iff $c$ is not balanced.
    \item
      $c$ is \emph{pseudo-balanced} iff every edge of $\tor(c)$ belongs to a directed cycle.
  \end{list}
\end{dfn}
Gil Kalai further asked whether the number $m$ can be given bounds in terms of
$\n$, and what is the result of demanding a ``non-trivial majority,'' e.g. 51\%.
We shall consider loose bounds while addressing the general
case, when voters are permitted to abstain.

To determine what symmetric sets of partial choice functions could produce an
arbitrary outcome, we shall characterize the set of all possible outcomes of a
symmetric set of choice functions, the ``majority closure'' $\majcl(\subC)$.
$\subC$ is a satisfactory class for this extension to McGarvey's theorem iff
$\majcl(\subC) = \C$.
\begin{dfn}
  \label{majdef}
  For $\subC \subseteq \C$ let $\majcl(\subC)$ be
  the set of $d \in \C$ such that, for some set of weights
  $\{ r_c \in [0,1]_\Q : c \in \subC \}$ with $\sum_{c \in \subC} r_c = 1$,
  \[
  d\{x,y\} = x \iff \sum_{c \in \subC} W^x_y(c) r_c > 0.
  \]
\end{dfn}
Here and throughout this paper, $[0,1]_\Q$ refers to the rationals in the unit
interval of the real line.  Note that we cannot assume that
$\majcl(\majcl(\subC)) = \majcl(\subC)$, and in fact we shall see this is not
true.  We must show that each function $d$ in the majority closure, thus
defined, can in fact be the outcome of a finite collection of voters using
choice functions in $\subC$; this is Claim~\ref{majthm}.

\citet{shelah2009mdp} gave a characterization of $\majcl(\subC)$ for symmetric
sets of full choice functions; there are just two cases.  $\majcl(\subC) = \C$
iff some $c \in \subC$ is imbalanced; if every $c \in \subC$ is balanced, then
$\majcl(\subC)$ is the set of all pseudo-balanced functions.  The situation
with possible abstention is more complicated; there are four cases.

In addition to the representations as relations, choice functions, or oriented
graphs discussed above, we shall also find it convenient to map each choice
function $c$ to a sequence in $[-1,1]_\Q$ indexed by $X^2$, the ``probability
sequence''
\[
\pr(c) = \langle W^x_y(c) : (x, y) \in X^2 \rangle.
\]
Let $\pr(\subC) = \{ \pr(c) : c \in \subC \}$.
$\pr(X)$ will denote the set of all sequences $\prt$ in $[-1,1]^{X^2}_\Q$ such that
$t_{x,y} = -t_{y,x}$; $\pr(X)$ contains $\pr(\C)$.

\newpage
\section{Basic Definitions and Facts}

\begin{dfn}
  \label{prtdef}
  For a probability sequence $\prt \in \pr(X)$,
  \begin{list}{(\alph{dlcount})}{\usecounter{dlcount}}
    \item
      $\prt$ is \emph{balanced} iff for each $x \in X$
      \[
      \sum_{y \in X} t_{x,y} = 0.
      \]
    \item
      $\maj(\prt)$ is the $c \in \C$ such that $c\{x,y\} = x \iff t_{x,y} > 0$.
  \end{list}
\end{dfn}
Note that $\maj$ and $\pr$ are mutually inverse functions from $\C$ to
$\pr(\C)$; $c = \maj(\pr(c))$ and $\prt = \pr(\maj(\prt))$ for all
$c \in \C$ and $\prt \in \pr(\C)$.
$\maj$ is also defined on the much larger set $\pr(X)$, which it maps onto $\C$.
The reason for reusing the name ``balanced'' is clear:
\begin{clm}
  \label{balanceclm}
  If $c \in \C$ is a balanced choice function, then $\pr(c)$ is a balanced
  probability sequence.  If $\prt \in \pr(\C)$ is a balanced probability
  sequence, then $\maj(\prt)$ is a balanced choice function.
\end{clm}
\begin{proof}
  Suppose $c \in \C$ is balanced.
  For every $x \in X$,
  \[
  \sum_{y \in X} W^x_y(c) = 0.
  \]
  $\pr(c)_{x,y} = W^x_y(c)$, so
  \[
  \sum_{y \in X} \pr(c)_{x,y} = 0,
  \]
  i.e.\ $\pr(c)$ is balanced.

  Now suppose $\prt \in \pr(\C)$ is balanced.  $\prt = \pr(c)$ for some
  $c \in \C$, and $c = \maj(\prt)$.
  \[
  \sum_{y \in X} t_{x,y} = 0,
  \]
  but $t_{x,y} = W^x_y(c)$, so
  \[
  \sum_{y \in X} W^x_y(c) = 0
  \]
  and $c$ is balanced.
\end{proof}
However, it is not the case that $\maj(\prt)$ is balanced for every balanced
$\prt \in \pr(X)$.  For instance, choose $x, z \in X$ and define
a sequence $\prt$ which has $t_{z,x} = 1$, $t_{x,z} = -1$; for all
$y \notin \{x,z\}$
\begin{align*}
  t_{x,y} = t_{y,z}& = \frac{1}{\n - 2},\\
  t_{y,x} = t_{z,y}& = \frac{-1}{\n - 2};
\end{align*}
and for all other pairs $(u, v)$, $t_{u,v} = 0$.  One may check that $\prt$ is
balanced, yet $\maj(\prt)$ has $\sum_{y \in X} W^x_y(\maj(\prt)) = \n - 3 > 0$
for any $X$ with at least 4 elements.

\begin{dfn}
  \label{choicedef}
  For a choice function $c \in \C$,
  \begin{list}{(\alph{dlcount})}{\usecounter{dlcount}}
    \item
      $c$ is \emph{partisan} iff there is nonempty $W \subsetneq X$ such that
      \[
      c\{x,y\} = x \iff x \in W \text{ and } y \notin W.
      \]
    \item
      $c$ is \emph{tiered} iff there is a partition $\{X_1, X_2, \ldots, X_k\}$ of $X$
      such that $c\{x,y\} = x$ iff $x \in X_i$,
      $y \in X_j$, and $i > j$.  We say $c$ is $k$-tiered where the partition has $k$
      sets.  (So a partisan function is 2-tiered.)
    \item
      $c$ is \emph{chaotic} iff it is both imbalanced and not partisan.
  \end{list}
\end{dfn}

\begin{dfn}
  \label{subdef}
  For a subset $\subC \subseteq \C$,
  \begin{list}{(\alph{dlcount})}{\usecounter{dlcount}}
    \item
      $\subC$ is \emph{trivial} iff $\subC = \0$ or $\subC = \{c\}$ where
      $\dom c = \0$, i.e.\ $c$ makes no decisions.
    \item
      $\subC$ is \emph{balanced} iff every $c \in \subC$ is balanced;
      $\subC$ is \emph{imbalanced} iff $\subC$ is not balanced.
    \item
      $\subC$ is \emph{partisan} iff every $c \in \subC$ is partisan;
      $\subC$ is \emph{nonpartisan} iff $\subC$ is not partisan.
    \item
      $\subC$ is \emph{chaotic} iff there is some chaotic $c \in \subC$.
    \item
      $\prcl(\subC)$ is the convex hull of $\pr(\subC)$, i.e.
      \[
      \prcl(\subC) = \left\{ \sum_{i=1}^k r_i \prt_i : k \in \N, r_i \in [0,1]_\Q, \sum_{i=1}^k r_i = 1,
      \prt_i \in \pr(\subC) \right\}.
      \]
  \end{list}
\end{dfn}

We can now establish some straightforward results which allow us to describe
the possible outcomes due to voters chosen from a given symmetric set more
explicitly.

\begin{clm}
  If $\subC$ is a symmetric subset of $\C$, then $d$ is the strict simple
  majority outcome of some finite set $\{c_1, \ldots, c_m\}$
  chosen from $\subC$ iff
  \[
  d\{x,y\} = x \iff \sum_{i = 1}^m W^x_y(c_i) > 0.
  \]
\end{clm}
\begin{proof}
  $d$ is the strict simple majority outcome iff
  \begin{align*}
    d\{x,y\} = x
    \iff & \abs{\{i : c_i\{x,y\} = x\}} > \abs{\{i : c_i\{x,y\} = y\}} \\
    \iff & \sum_{i=1}^m W^x_y(c_i) > 0.
  \end{align*}
\end{proof}

\begin{clm}
  \label{majthm}
  Suppose $\subC$ is a symmetric subset of $\C$.  Then $d \in \majcl(\subC)$ iff
  $d$ is a strict simple majority outcome of some
  $\{c_1, \ldots, c_m\} \subseteq \subC$.
\end{clm}
\begin{proof}
  If $d \in \C$ and there is a set $\{c_1, \ldots, c_m\} \subseteq \subC$ with
  $d$ as their strict simple majority outcome, then for each
  $c \in \subC$ let $r_c$ be the number of times that $c$ appears among the $c_i$,
  divided by $m$:
  \[
  r_c = \frac{\abs{\{i : c_i = c\}}}{m}.
  \]
  Clearly $\forall c \in \subC \  0 \leq r_c \leq 1$ and
  $\sum_{c \in \subC} r_c = 1$.  Furthermore,
  \[
  d\{x,y\} = x \iff
  \sum_{i = 1}^m W^x_y(c_i) > 0 \iff
  \sum_{i = 1}^m \frac{W^x_y(c_i)}{m} > 0 \iff
  \sum_{c \in \subC} W^x_y(c) r_c > 0.
  \]

  Conversely, suppose $d \in \majcl(\subC)$.  There are $r_c = \frac{a_c}{b_c}$ for
  each $c \in \subC$ with $\sum r_c = 1$, and
  \[
  d\{x,y\} = x \iff \sum_{c \in \subC} W^x_y(c) r_c > 0.
  \]
  Let $m = \lcm\{b_c : c \in \subC\}$.
  Now construct a finite set $\{c_1, \ldots, c_m\}$ by taking $a_c \frac{m}{b_c}$
  copies of each $c \in \subC$.  $\frac{m}{b_c}$ is an integer since $b_c \mid m$.
  Now for each $c \in \subC$,
  \[
  \sum_{c_i = c} W^x_y(c_i) = W^x_y(c) \cdot \abs{\{i : c_i = c\}} =
  W^x_y(c) a_c \frac{m}{b_c} = W^x_y(c) r_c m.
  \]
  So
  \begin{align*}
    \sum_{i = 1}^m W^x_y(c_i) = \sum_{c \in \subC} \sum_{c_i = c} W^x_y(c_i) =
    \sum_{c \in \subC} W^x_y(c) r_c m. \\
    \sum_{c \in \subC} W^x_y(c) r_c m > 0 \iff
    \sum_{c \in \subC} W^x_y(c) r_c > 0,
  \end{align*}
  so $d$ is the simple majority relation of the $c_i$.
\end{proof}

\begin{clm}
  \label{majclisprcl}
  For any symmetric $\subC \subseteq \C$,
  $\majcl(\subC) = \{\maj(\prt) : \prt \in \prcl(\subC)\}$.
\end{clm}
\begin{proof}
  Suppose $d \in \majcl(\subC)$.  Then there are $r_c$ per $c \in \subC$ with
  $\sum_{c \in \subC}r_c = 1$ and
  \[
  d\{x,y\} = x \iff \sum_{c \in \subC} r_c W^x_y(c) > 0.
  \]
  Let $\prt = \sum_{c \in \subC} r_c \pr(c)$ be a probability sequence
  in $\prcl(\subC)$.  Then
  \[
  t_{x,y} = \sum_{c \in \subC} r_c \pr(c)_{x,y} =
  \sum_{c \in \subC} r_c W^x_y(c).
  \]
  So
  \[
  \maj(\prt)\{x,y\} = x \iff \sum_{c \in \subC} r_c W^x_y(c) > 0,
  \]
  i.e.\ $d = \maj(\prt)$.

  Suppose $d = \maj(\prt)$ for some $\prt \in \prcl(\subC)$.  Then
  for some $r_c$ with $\sum_{c \in \subC} r_c = 1$,
  $d = \maj\left(\sum_{c \in \subC} r_c \pr(c)\right)$.
  \[
  d\{x,y\} = x \iff \sum_{c \in \subC} r_c \pr(c)_{x,y} > 0 \iff
  \sum_{c \in \subC} r_c W^x_y(c) > 0.
  \]
  So $d \in \majcl(\subC)$.
\end{proof}

\newpage
\section{A Characterization}
\begin{main}
  Given a symmetric $\subC \subseteq \C$,
  \begin{enumerate}
    \item
      \label{trivial}
      $\subC$ is trivial $\iff \majcl(\subC)$ is trivial.
    \item
      \label{balanced}
      $\subC$ is balanced but nontrivial $\iff \majcl(\subC)$ is all
      pseudo-balanced choice functions on $X$.
    \item
      \label{partisan}
      $\subC$ is partisan and nontrivial $\iff \majcl(\subC)$ is all
      tiered choice functions on $X$.
    \item
      \label{else}
      $\subC$ is imbalanced and nonpartisan $\iff \majcl(\subC) = \C$.
  \end{enumerate}
\end{main}
\begin{proof}
  No nontrivial $\subC$ is both balanced and partisan, so the sets
  \begin{gather*}
    \{\subC \subset \C : \subC \text{ is trivial} \},\\
    \{\subC \subset \C : \subC \text{ is symmetric, nontrivial, and balanced} \},\\
    \{\subC \subset \C : \subC \text{ is symmetric, nontrivial, and partisan} \}, \text{ and}\\
    \{\subC \subseteq \C : \subC \text{ is symmetric, imbalanced, and nonpartisan}\}
  \end{gather*}
  partition the family of all symmetric subsets of $\C$.
  Thus proving each forward
  implication in the claim will give the reverse implications.

  \ref{trivial}.
  If no-one in $\subC$ makes any choices, no combination can have a
  majority choice.

  \ref{balanced}.
  This is the content of Section~\ref{balancesec}, below.

  \ref{partisan}.
  Suppose $\subC$ is partisan and $d \in \majcl(\subC)$.
  Then there is a finite set $\{c_1, \ldots, c_m\} \subseteq \subC$
  with $d$ as the strict simple majority outcome.
  For each $x \in X$, let $k_x$ be the number of $c_i$ such that $x$ is
  in the ``winning'' partite set.
  For any $(x,y) \in X^2$, let $k$ be the number of
  $c_i$ where $x$ and $y$ are both in the winning subset;
  the number of $c_i$ with $c_i\{x,y\} = x$ is $k_x - k$, and the number with
  $c_i\{x,y\} = y$ is $k_y - k$, so
  \[
  d\{x,y\} = x \iff k_x - k > k_y - k \iff k_x > k_y.
  \]
  Partition $X$ into subsets for each value of $k_x$,
  \[
  X = \bigcup \{\{x \in X : k_x = k\} : k \in \{k_x : x \in X \}\}.
  \]
  These subsets form the tiers, so $d$ is a tiered function.

  For the reverse inclusion, suppose $d$ is an arbitrary tiered function
  with tiers $\{X_1, \ldots, X_k\}$.  Let $c \in \subC$.
  For each $x \in X$, let $\Gamma_x \subset \perX$ be all permutations
  holding $x$ fixed and $c_x$ be a choice function in $\subC$ with $x$
  in its winning subset.
  (If $y$ is in the winning subset of $c$, $c^{(x,y)}$ will suffice for $c_x$.)
  Now construct $\{c_1, \ldots, c_m\}$ by taking,
  for $1 \leq i \leq k$, each $x \in X_i$,
  and each $\sigma \in \Gamma_x$, $i$ copies of $c^\sigma_x$.
  Suppose $x \in X_i$, $y \in X_j$, and $i > j$ (so $d\{x,y\} = x$).  Then
  for any $z \in X_i \setminus \{x\}$, any $z \in X_j \setminus \{y\}$, and any $z$ in
  other tiers, $x$ and $y$ are chosen by equal numbers of
  $\{c^\sigma_z : \sigma \in \Gamma_z\}$.
  So they are chosen by an equal number of $\{c_1, \ldots, c_m\}$ except among the
  functions $c^\sigma_x$ and $c^\sigma_y$; there are $i$ such functions
  $c^\sigma_x$ for each $\sigma \in \Gamma_x$, and $j$ such functions $c^\sigma_y$
  for each $\sigma \in \Gamma_y$.  The number of permutations fixing $x$ and the
  number fixing $y$ are the same, $\abs{\Gamma_x} = \abs{\Gamma_y}$.  The number
  of permutations fixing $x$ and putting $y$ in the losing subset and the number
  fixing $y$ and putting $x$ in the losing subset is the same, $l$.  There are $i
  l$ functions which choose $x$ over $y$, and $j l$ which choose $y$ over $x$; $i
  l > j l$ so the strict simple majority outcome of the $c_i$ chooses $x$.

  If $i = j$, then $i l = j l$, so $x$ and $y$ tie in the majority outcome.
  Thus $d$ is the strict simple
  majority outcome of $\{c_1, \ldots, c_m\}$ and $d \in \majcl(\subC)$.

  \ref{else}.
  If $\subC$ is a chaotic set, the conclusion follows from Section~\ref{elsesec}, below.

  Otherwise, $\subC$ is imbalanced and nonpartisan, but not chaotic.
  In this case, $\subC$ contains imbalanced functions which are all partisan,
  and nonpartisan functions which are all balanced.
  By symmetry, $\subC$ contains a nontrivial balanced symmetric subset $\subB$ and a
  nontrivial partisan symmetric subset $\subP$.

  Let $d \in \C$.  Suppose $b \to a$ is any edge of $d$.  For each $z \in X
  \setminus \{a,b\}$, by Claim~\ref{triclm} there is a voter population $T_z$ from
  $\subB$ such that $a$ beats $z$, $z$ beats $b$, $b$ beats a, and no other pairs
  are decided; moreover, the same number of voters, say $m$, choose the winner in
  each pair, and no dissenting votes occur.
  In the combination $\bigcup_z T_z$, $m$ voters choose $a$ over each $z$,
  and each $z$ over $b$.  $m(\n - 2)$ voters choose $b$ over $a$.

  Let $c \in \subP$ be a partisan function, and $l$ be the size of the set of
  winning candidates under $c$.  Let $A \subset \subP$ be the set of functions
  symmetric to $c$ so that $a$ is on the losing side, $B \subset \subP$ be the
  functions symmetric to $c$ so that $b$ is on the winning side, and $C = A \cap
  B$---the permutations with $b$ winning and $a$ losing.
  \[
  \abs{A} = \binom{\n - 1}{l}, \quad
  \abs{B} = \binom{\n - 1}{l - 1}, \quad
  \abs{C} = \binom{\n - 2}{l - 1}.
  \]
  Any candidates besides $a$ or $b$ are tied over all the voters of $A$, $B$, or
  $C$, since each is selected in the winning set an equal number of times.
  Furthermore, $a$ does not beat $b$ in any choice function in $A$, $B$, or $C$.

  Now we seek to take appropriate numbers of copies of the sets $A$, $B$, and $C$
  so that, for some constant $k$, any $z \in X \setminus \{a,b\}$
  defeats $a$ by $k$ votes, and loses to $b$ by $k$ votes.  Say we have $k_0$
  copies of $A$, $k_1$ copies of $B$, and $k_2$ copies of $C$ in a population $D$.
  Now for each $z \in X \setminus \{a, b\}$,
  \begin{align*}
    \abs{\{c \in D : c\{z,a\} = z\}}& = k_0 \binom{\n - 2}{l - 1} + k_1 \binom{\n -
    3}{l - 2} + k_2 \binom{\n - 3}{l - 2}, \\
    \abs{\{c \in D : c\{z,a\} = a\}}& = k_1 \binom{\n - 3}{l - 2}, \\
    \abs{\{c \in D : c\{b,z\} = b\}}& = k_0 \binom{\n - 3}{l - 1} + k_1 \binom{\n -
    2}{l - 1} + k_2 \binom{\n - 3}{l - 1}, \\
    \abs{\{c \in D : c\{b,z\} = z\}}& = k_0 \binom{\n - 3}{l - 1}.
  \end{align*}
  So we solve $k_0 \binom{\n - 2}{l - 1} + k_2 \binom{\n - 3}{l - 2} = k_1
  \binom{\n - 2}{l - 1} + k_2 \binom{\n - 3}{l - 1}$, with coefficients not all
  zero.  There are solutions
  \begin{align*}
    k_0& = \n - 2l,& k_1& = 0,& k_2& = \n - 2, & \text{if } l \leq \frac{\n}{2}, \\
    k_0& = 0,& k_1& = 2l - \n,& k_2& = \n - 2, & \text{if } l > \frac{\n}{2}.
  \end{align*}
  Let $k = k_0 \binom{\n - 2}{l - 1} + k_2 \binom{\n - 3}{l - 2}$, the number of
  votes for $b$ over $z$, or $z$ over $a$, for any $z \in X \setminus \{a,b\}$.
  Now take the union of $k$ copies of the population $\bigcup_z T_z$, $m k_0$ copies of $A$,
  $m k_1$ copies of $B$, and $m k_2$ copies of $C$ for our voter population.
  $b$ defeats any $z$ $m k$ times among the partisan functions, and $z$
  defeats $b$ $k m$ times among the balanced functions, so they tie; similarly,
  $a$ and $z$ tie.  Thus we are left with $b \to a$.  The union of such
  populations for each edge in $d$ yields $d$ as the majority outcome.

  Hence $\C \subseteq \majcl(\subC)$.
\end{proof}

\newpage
\section{Balanced Choices}
\label{balancesec}
\begin{dfn}
  For a choice function $c \in \C$,
  \begin{list}{(\alph{dlcount})}{\usecounter{dlcount}}
    \item
      $c$ is \emph{triangular} iff for some $\{x, y, z\} \in \binom{X}{3}$,
      $c\{x,y\} = x$, $c\{y,z\} = y$, $c\{z,x\} = z$,
      and for any $\{u, v\} \not \subset \{x, y, z\}$, $\{u, v\} \notin \dom c$.
      We write $c^{x,y,z}$ for such $c$.
    \item
      $c$ is \emph{cyclic} iff for some $\{x_1, \ldots, x_k\} \in \binom{X}{k}$,
      \begin{itemize}
        \item
          $c\{x_i,x_j\} = x_i$ iff $j \equiv i + 1$ mod $k$.
        \item
          No other pair $\{x, y\}$ is in $\dom c$.
      \end{itemize}
      We write $c^{x_1, \ldots, x_k}$ for such $c$.  (Triangular functions
      are a specific case of cyclic functions.)
    \item
      Let $\prbl$ be the set of all balanced $\prt \in \pr(X)$.
  \end{list}
\end{dfn}
Note that if $\subC$ is symmetric and $c^{x,y,z} \in \subC$, then
$c^{u,v,w} \in \subC$ for any $\{u, v, w\} \in \binom{X}{3}$; similarly for cyclic functions.

\begin{clm}
  \label{weightpseudo}
  If a choice function $c = \maj(\prt)$ for some $\prt \in \prbl$,
  then $c$ is pseudo-balanced.
\end{clm}
\begin{proof}
  Assume $c = \maj(\prt)$ for some $\prt \in \prbl$.
  Let $x \to y$ be any edge of $\tor(c)$; then $t_{x,y} > 0$.  Suppose $x \to y$ is
  in no directed cycle.  Let $Y$ be the set of $z \in X$ with a winning chain to
  $x$, i.e.
  \[
  Y = \bigcup\left\{\{z_1, \ldots, z_k\} \in \binom{X \setminus \{y\}}{k} :
  k < \n, \forall i < k \  c\{z_i, z_{i+1}\} = z_i,  z_k = x \right\}.
  \]
  Let $\overline{Y} = X \setminus Y$.  Note that $y \in \overline{Y}$ and $x \in
  Y$, so $Y$ and $\overline{Y}$ partition $X$ into nonempty sets.  Suppose $z \in
  Y$ and $v \in \overline{Y}$.
  Say $z, z_1, \ldots, z_k$ is a winning chain from $z$ to $x$.  If $t_{z,v} < 0$,
  then $v, z, z_1, \ldots, z_k$ forms a winning chain from $v$ to $x$.  If $v$ is
  $y$, the chain forms a directed cycle with the edge $x \to y$ and we are done.
  If $v \neq y$, this contradicts that $v \notin Y$.  So we assume that $t_{z,v} \geq 0$.

  Now
  \[
  \sum_{z \in Y, v \in \overline{Y}} t_{z,v} > 0,
  \]
  since every $t_{z,v} \geq 0$, and $t_{x,y} > 0$ is among them.
  For each $u \in \overline{Y}$ let
  \begin{align*}
    r_u = & \sum_{z \in Y} t_{z, u}, \\
    \bar{r}_u = & \sum_{v \in \overline{Y}} t_{v, u}.
  \end{align*}
  Since $\prt$ is balanced, $r_u + \bar{r}_u = 0$, so
  \[
  \sum_{u \in \overline{Y}} (r_u + \bar{r}_u) = 0
  = \sum_{u \in \overline{Y}} r_u + \sum_{u \in \overline{Y}} \bar{r}_u.
  \]
  We have seen that the first summand is positive, so the second summand is
  negative.  But it is zero because for each pair $(u, v) \in \overline{Y}^2$ we
  have $t_{u,v} + t_{v,u} = 0$.
  This is a contradiction; so $x \to y$ must be in some directed cycle.
\end{proof}

\begin{clm}
  \label{balanceconvex}
  $\prbl$ is a convex subset of $\pr(X)$.
\end{clm}
\begin{proof}
  Suppose $\prt$ and $\prs$ are balanced probability sequences, and
  $a \in [0,1]_\Q$.
  Then $\pru = a \prt + (1 - a) \prs$ has, for any $x \in X$,
  \begin{align*}
    \sum_{y \in X} u_{x,y} & = \sum_{y \in X} \left(a t_{x,y} + (1 - a) s_{x,y}\right) \\
    & = a \sum_{y \in X} t_{x,y} + (1 - a) \sum_{y \in X} s_{x,y} \\
    & = a 0 + (1 - a) 0 \  = \  0.
  \end{align*}
  So $\pru$ is a balanced probability sequence.
\end{proof}

\begin{clm}
  \label{majinpseud}
  If $\subC \subset \C$ is symmetric, balanced, and nontrivial,
  then every $d \in \majcl(\subC)$ is pseudo-balanced.
\end{clm}
\begin{proof}
  By Claim~\ref{balanceclm}, every $\prt \in \pr(\subC)$ is balanced.
  So by Claim~\ref{balanceconvex}, the convex hull $\prcl(\subC)$ is contained
  in $\prbl$.
  Any $d \in \majcl(\subC)$ is $\maj(\prt)$ for some $\prt \in \prcl(\subC)$, by
  Claim~\ref{majclisprcl}, hence is pseudo-balanced, by Claim~\ref{weightpseudo}.
\end{proof}

For the reverse inclusion, that every pseudo-balanced function is in the
majority closure, we shall consider the graph interpretation.
First we shall show that every triangular function is in the majority closure
of a balanced symmetric set.
\begin{figure}
  \hfill
  \begin{tikzpicture}[>=latex,inner sep=3pt]
  \useasboundingbox (0dd,-4dd) rectangle (88dd,76dd);
  \node (x)  at (22dd,76dd) {$x$};
  \node (y)  at ( 0dd,38dd) {$y$};
  \node (z)  at (22dd, 0dd) {$z$};
  \node (4)  at (66dd, 0dd) [fill,circle,inner sep=1.5pt] {};
  \node (r1) at (78dd,20dd) {};
  \node at (80dd,24dd) [fill=black!70!white,circle,inner sep=1pt] {};
  \node at (85dd,33dd) [fill=black!70!white,circle,inner sep=1pt] {};
  \node at (85dd,43dd) [fill=black!70!white,circle,inner sep=1pt] {};
  \node at (80dd,52dd) [fill=black!70!white,circle,inner sep=1pt] {};
  \node (r3) at (78dd,56dd) {};
  \node (m)  at (66dd,76dd) [fill,circle,inner sep=1.5pt] {};
  \draw [->] (x) -- (y);
  \draw [->] (y) -- (z);
  \draw [->] (z) -- (4);
  \draw [->] (4) -- (r1);
  \draw [->] (r3) -- (m);
  \draw [->] (m) -- (x);
\end{tikzpicture}
  \hfill
  \begin{tikzpicture}[>=latex,inner sep=3pt]
  \useasboundingbox (0dd,-4dd) rectangle (88dd,76dd);
  \node (x)  at (22dd,76dd) {$x$};
  \node (y)  at ( 0dd,38dd) {$y$};
  \node (z)  at (22dd, 0dd) {$z$};
  \node (4)  at (66dd, 0dd) [fill,circle,inner sep=1.5pt] {};
  \node (r1) at (78dd,20dd) {};
  \node at (80dd,24dd) [fill=black!70!white,circle,inner sep=1pt] {};
  \node at (85dd,33dd) [fill=black!70!white,circle,inner sep=1pt] {};
  \node at (85dd,43dd) [fill=black!70!white,circle,inner sep=1pt] {};
  \node at (80dd,52dd) [fill=black!70!white,circle,inner sep=1pt] {};
  \node (r3) at (78dd,56dd) {};
  \node (m)  at (66dd,76dd) [fill,circle,inner sep=1.5pt] {};
  \draw [->] (x) -- (4);
  \draw [->] (y) -- (z);
  \draw [->] (z) -- (x);
  \draw [->] (4) -- (r1);
  \draw [->] (r3) -- (m);
  \draw [->] (m) -- (y);
\end{tikzpicture}
  \hfill
  \begin{tikzpicture}[>=latex,inner sep=3pt]
  \useasboundingbox (0dd,-4dd) rectangle (88dd,76dd);
  \node (x)  at (22dd,76dd) {$x$};
  \node (y)  at ( 0dd,38dd) {$y$};
  \node (z)  at (22dd, 0dd) {$z$};
  \node (4)  at (66dd, 0dd) [fill,circle,inner sep=1.5pt] {};
  \node (r1) at (78dd,20dd) {};
  \node at (80dd,24dd) [fill=black!70!white,circle,inner sep=1pt] {};
  \node at (85dd,33dd) [fill=black!70!white,circle,inner sep=1pt] {};
  \node at (85dd,43dd) [fill=black!70!white,circle,inner sep=1pt] {};
  \node at (80dd,52dd) [fill=black!70!white,circle,inner sep=1pt] {};
  \node (r3) at (78dd,56dd) {};
  \node (m)  at (66dd,76dd) [fill,circle,inner sep=1.5pt] {};
  \draw [->] (x) -- (y);
  \draw [->] (y) -- (4);
  \draw [->] (z) -- (x);
  \draw [->] (m) -- (z);
  \draw [->] (r3) -- (m);
  \draw [->] (4) -- (r1);
\end{tikzpicture}
  \hfill

  \bigskip
  \hfill
  \begin{tikzpicture}[>=latex,inner sep=3pt]
  \useasboundingbox (0dd,0dd) rectangle (88dd,80dd);
  \node (x)  at (22dd,76dd) {$x$};
  \node (y)  at ( 0dd,38dd) {$y$};
  \node (z)  at (22dd, 0dd) {$z$};
  \node (4)  at (66dd, 0dd) [fill,circle,inner sep=1.5pt] {};
  \node (r1) at (78dd,20dd) {};
  \node at (80dd,24dd) [fill=black!70!white,circle,inner sep=1pt] {};
  \node at (85dd,33dd) [fill=black!70!white,circle,inner sep=1pt] {};
  \node at (85dd,43dd) [fill=black!70!white,circle,inner sep=1pt] {};
  \node at (80dd,52dd) [fill=black!70!white,circle,inner sep=1pt] {};
  \node (r3) at (78dd,56dd) {};
  \node (m)  at (66dd,76dd) [fill,circle,inner sep=1.5pt] {};
  \draw [->] (x) -- (y);
  \draw [->] (y) -- (z);
  \draw [->] (z) -- (m);
  \draw [->] (m) -- (r3);
  \draw [->] (r1) -- (4);
  \draw [->] (4) -- (x);
\end{tikzpicture}
  \hfill
  \begin{tikzpicture}[>=latex,inner sep=3pt]
  \useasboundingbox (0dd,0dd) rectangle (88dd,80dd);
  \node (x)  at (22dd,76dd) {$x$};
  \node (y)  at ( 0dd,38dd) {$y$};
  \node (z)  at (22dd, 0dd) {$z$};
  \node (4)  at (66dd, 0dd) [fill,circle,inner sep=1.5pt] {};
  \node (r1) at (78dd,20dd) {};
  \node at (80dd,24dd) [fill=black!70!white,circle,inner sep=1pt] {};
  \node at (85dd,33dd) [fill=black!70!white,circle,inner sep=1pt] {};
  \node at (85dd,43dd) [fill=black!70!white,circle,inner sep=1pt] {};
  \node at (80dd,52dd) [fill=black!70!white,circle,inner sep=1pt] {};
  \node (r3) at (78dd,56dd) {};
  \node (m)  at (66dd,76dd) [fill,circle,inner sep=1.5pt] {};
  \draw [->] (x) -- (m);
  \draw [->] (y) -- (z);
  \draw [->] (z) -- (x);
  \draw [->] (m) -- (r3);
  \draw [->] (r1) -- (4);
  \draw [->] (4) -- (y);
\end{tikzpicture}
  \hfill
  \begin{tikzpicture}[>=latex,inner sep=3pt]
  \useasboundingbox (0dd,0dd) rectangle (88dd,80dd);
  \node (x)  at (22dd,76dd) {$x$};
  \node (y)  at ( 0dd,38dd) {$y$};
  \node (z)  at (22dd, 0dd) {$z$};
  \node (4)  at (66dd, 0dd) [fill,circle,inner sep=1.5pt] {};
  \node (r1) at (78dd,20dd) {};
  \node at (80dd,24dd) [fill=black!70!white,circle,inner sep=1pt] {};
  \node at (85dd,33dd) [fill=black!70!white,circle,inner sep=1pt] {};
  \node at (85dd,43dd) [fill=black!70!white,circle,inner sep=1pt] {};
  \node at (80dd,52dd) [fill=black!70!white,circle,inner sep=1pt] {};
  \node (r3) at (78dd,56dd) {};
  \node (m)  at (66dd,76dd) [fill,circle,inner sep=1.5pt] {};
  \draw [->] (x) -- (y);
  \draw [->] (y) -- (m);
  \draw [->] (z) -- (x);
  \draw [->] (m) -- (r3);
  \draw [->] (r1) -- (4);
  \draw [->] (4) -- (z);
\end{tikzpicture}
  \hfill
  \caption{Constructing a triangle from permutations of a cycle, as in Claim~\ref{triclm}.  Note
  that every edge is matched by an opposing one, except for $xy$, $yz$, and $zx$.}
  \label{trifig}
\end{figure}
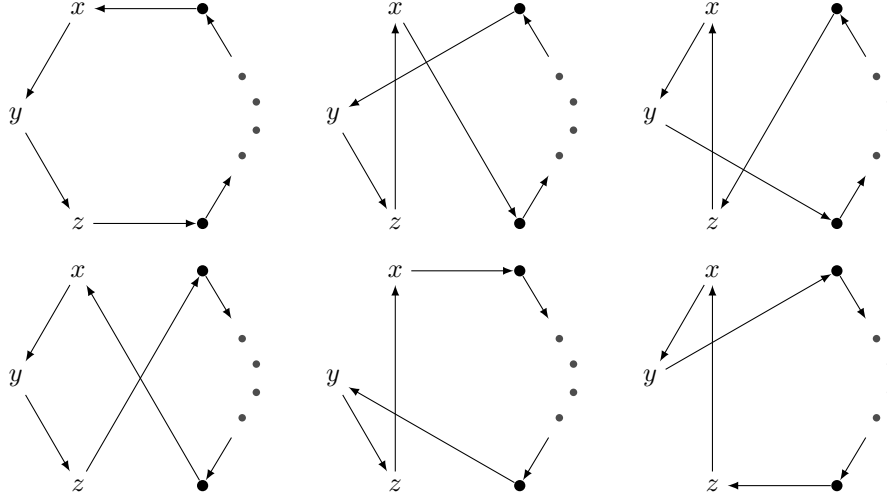
\begin{clm}
  \label{triclm}
  If $\subC \subset \C$ is symmetric, balanced, and nontrivial,
  then for any $\{x, y, z\} \in \binom{X}{3}$, the triangular function
  $c^{x,y,z} \in \majcl(\subC)$.
  Moreover, there is a voter population generating $c^{x,y,z}$ such that the
  same number of votes choose $x$ over $y$, $y$ over $z$, or $z$ over $x$, and
  there are no opposing votes in any of these cases.
\end{clm}
\begin{proof}
  Let $c_0 \in \subC$ and let $m$ be the size of the smallest directed cycle
  in $\tor(c_0)$.  Index the distinct elements of this cycle, in order, as
  $x_1, x_2, \ldots, x_m$.
  Let $c = c_0^\sigma$ where $\sigma \in \perX$ takes $x_1$ to $x$,
  $x_2$ to $y$, and $x_3$ to $z$.
  Now consider the cycle in $\tor(c)$, so $x_1 = x$, $x_2 = y$, and $x_3 = z$.
  We define
  \begin{align*}
    c_1& = c, \\
    c_2& = c^\rho \text{ where } \rho \text{ takes } x \mapsto y, y \mapsto z,
    z \mapsto x, \\
    c_3& = c^\rho \text{ where } \rho \text{ takes } x \mapsto z, y \mapsto x,
    z \mapsto y.
  \end{align*}
  Let $\Gamma_{x,y,z} \subset \perX$ be the permutations fixing $x$,
  $y$, and $z$, and take $\{ c_i^\sigma : 1\nobreak\leq\nobreak
  i\nobreak\leq\nobreak3, \, \sigma \in \Gamma_{x,y,z} \}$ as our finite set of
  voters.  We claim the majority outcome of this set is $c^{x,y,z}$.

  If $u, v \in X \setminus \{x, y, z\}$, then equal numbers of permutations
  $\sigma \in \Gamma_{x,y,z}$ have $c_i^\sigma\{u,v\} = u$ and $c_i^\sigma\{u,v\}
  = v$, so they are tied.  If $u$ is among $\{x, y, z\}$ and $v$ is not, then
  there are some out-edges from $u$ in $\tor(c)$.  Since $c$ is balanced, there
  are an equal number of in-edges to $u$.  For each $i \in \{1,2,3\}$, $v$ will
  occupy the in-edges which are not on the cycle in as many permutations of $c_i$
  as permutations where it occupies out-edges which are not on the cycle.  If $m =
  3$, then $v$ occupies no edges on the cycle.  Otherwise, $v$ occupies an
  out-edge from $u$ along the cycle only in
  \begin{itemize}
    \item
      $c_1^\sigma$, for some set of $\sigma \in \Gamma_{x,y,z}$, if $u = z$;
      then $v$ occupies an in-edge to $z$ on the cycle in $c_3$ for an equal
      number of permutations.
    \item
      $c_2^\sigma$, for some set of $\sigma \in \Gamma_{x,y,z}$, if $u = x$;
      then $v$ occupies an in-edge to $x$ on the cycle in $c_1$ for an equal
      number of permutations.
    \item
      $c_3^\sigma$, for some set of $\sigma \in \Gamma_{x,y,z}$, if $u = y$;
      then $v$ occupies an in-edge to $y$ on the cycle in $c_2$ for an equal
      number of permutations.
  \end{itemize}
  Thus $u$ and $v$ are tied.
  Otherwise, $\{u, v\} \subset \{x, y, z\}$.
  For each $i \in \{1,2,3\}$, there are $\abs{\Gamma_{x,y,z}}$ many $c_i^\sigma$.
  $c_i^\sigma\{x,y\} = x$ if $i = 1$ or $i = 3$; this is two-thirds of the
  voters, so $x$ beats $y$.
  Similarly, $c_i^\sigma\{y,z\} = y$ if $i = 1$ or $i = 2$, and
  $c_i^\sigma\{z,x\} = z$ if $i = 2$ or $i = 3$.
  Thus the strict simple majority outcome is $c^{x,y,z}$.

  The collection $\{c_i^\sigma : 1 \leq i \leq 3, \sigma \in \Gamma_{x,y,z}\}$ is
  the voter population in the claim.  There are no opposing votes between $x$,
  $y$, or $z$ because such a vote would imply an edge between two vertices of the
  cycle which is not itself on the cycle, contradicting our choice of the smallest
  cycle.
\end{proof}

\begin{figure}
  \begin{tikzpicture}[>=latex,inner sep=3pt,minimum size=2.8pt]
  \node (x0) at ( 1cc,162dd) {$x_0$};
  \node (x1) at (-2cc, 11cc) {$x_1$};
  \node (x2) at ( 1cc,102dd) {$x_2$};
  \draw [->] (x0) -- (x1);
  \draw [->] (x1) -- (x2);
  \draw [->] (x2) -- (x0);

  \node (y0) at (5cc,162dd) {$x_0$};
  \node (y2) at (5cc,102dd) {$x_2$};
  \node (y3) at (8cc, 11cc) {$x_3$};
  \draw [->] (y0) -- (y2);
  \draw [->] (y2) -- (y3);
  \draw [->] (y3) -- (y0);

  \path (x1) -- node {\Huge $+$} (y3);

  \node (z0) at (15cc,14cc) {$x_0$};
  \node (z1) at (12cc,11cc) {$x_1$};
  \node (z2) at (15cc, 8cc) {$x_2$};
  \node (z3) at (18cc,11cc) {$x_3$};
  \draw [->] (z0) -- (z1);
  \draw [->] (z1) -- (z2);
  \draw [->] (z2) -- (z3);
  \draw [->] (z3) -- (z0);

  \path (y3) -- node {\Huge $\mapsto$} (z1);

  \node (u0) at ( 1cc,68dd) {$x_0$};
  \node (u1) at (-2cc, 4cc) {$x_1$};
  \node (ub) at (-23dd,25dd) {};
  \node      at (-20dd,23dd) [fill,circle,inner sep=0pt] {};
  \node      at (-16dd,19dd) [fill,circle,inner sep=0pt] {};
  \node      at (-1cc,15dd) [fill,circle,inner sep=0pt] {};
  \node (ue) at (-11dd,14dd) {};
  \node (ul) at (1cc,4dd) {$x_{m-1}$};
  \draw [->] (u0) -- (u1);
  \draw [->] (u1) -- (ub);
  \draw [->] (ue) -- (ul);
  \draw [->] (ul) -- (u0);

  \node (v0) at (5cc,68dd) {$x_0$};
  \node (vl) at (5cc,4dd) {$x_{m-1}$};
  \node (vm) at (8cc,3cc) {$x_m$};
  \draw [->] (v0) -- (vl);
  \draw [->] (vl) -- (vm);
  \draw [->] (vm) -- (v0);

  \node at (3cc,3cc) {\Huge $+$};
 
  \node (w0) at ( 15cc, 6cc) {$x_0$};
  \node (w1) at ( 12cc, 4cc) {$x_1$};
  \node (wb) at (145dd,25dd) {};
  \node      at (148dd,23dd) [fill,circle,inner sep=0pt] {};
  \node      at (152dd,19dd) [fill,circle,inner sep=0pt] {};
  \node      at (156dd,15dd) [fill,circle,inner sep=0pt] {};
  \node (we) at (157dd,14dd) {};
  \node (wl) at ( 15cc, 0cc) {$x_{m-1}$};
  \node (wm) at ( 18cc, 3cc) {$x_m$};
  \draw [->] (w0) -- (w1);
  \draw [->] (w1) -- (wb);
  \draw [->] (we) -- (wl);
  \draw [->] (wl) -- (wm);
  \draw [->] (wm) -- (w0);

  \path (vm) -- node[below] {\Huge $\mapsto$} (w1);
\end{tikzpicture}
  \caption{Constructing a cycle from triangles, as in Claim~\ref{cycclm}.}
  \label{cycfig}
\end{figure}
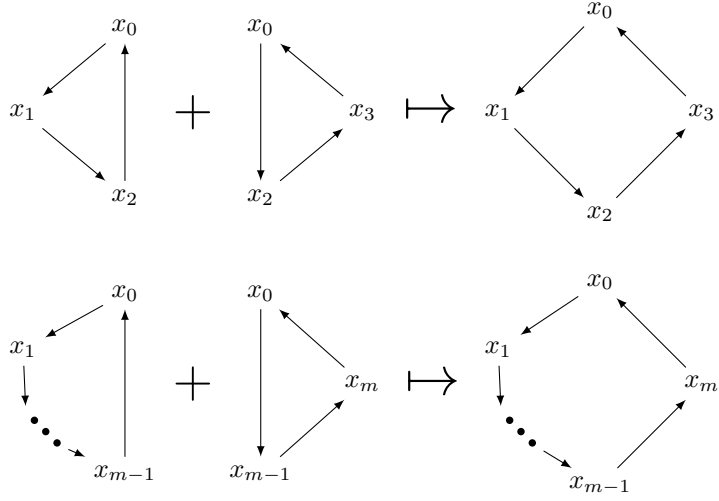
\begin{clm}
  \label{cycclm}
  If $\subC \subseteq \C$ and $c^{x,y,z} \in \majcl(\subC)$ for any $\{x, y, z\}
  \in \binom{X}{3}$, then for any $k \geq 3$ and $\{x_1, \ldots, x_k\} \in
  \binom{X}{k}$, $c^{x_1, \ldots, x_k} \in \majcl(\subC)$.
\end{clm}
\begin{proof}
  If $k = 3$, we have $c^{x_1, x_2, x_3} \in \majcl(\subC)$ by assumption.

  Suppose that $3 < k \leq \n$ and for any $k - 1$ distinct candidates
  $\{x_1, \ldots, x_{k-1}\} \subset X$ we have
  $c^{x_1, \ldots, x_{k - 1}} \in \majcl(\subC)$.  Say $\{c_1, \ldots, c_{m_0}\}$
  is a set of $m_0$ voters with $c^{x_1, \ldots, x_{k - 1}}$ as their strict
  simple majority outcome.  Let
  \[
  l_0 = \sum_{i=1}^{m_0} W_{x_1}^{x_{k-1}}(c_i).
  \]
  $l_0 > 0$ since $c^{x_1, \ldots, x_{k - 1}}\{x_1, x_{k-1}\} = x_{k-1}$.
  By hypothesis $c^{x_1,x_{k-1},x_k} \in \majcl(\subC)$.
  Say $\{c'_1, \ldots, c'_{m_1}\}$ is a set of $m_1$ voters with
  $c^{x_1,x_{k-1},x_k}$ as their strict simple majority outcome.
  Let
  \[
  l_1 = \sum_{i=1}^{m_1} W^{x_1}_{x_{k-1}}(c'_i).
  \]
  Like $l_0$, $l_1$ is a positive integer.
  Let $L = \lcm(l_0, l_1)$, and take $\frac{L}{l_0}$ copies of each $c_i$, and
  $\frac{L}{l_1}$ copies of each $c'_i$, to make a finite set of voters
  $\{d_1, \ldots, d_m\}$ where $m = \frac{L}{l_0} m_0 + \frac{L}{l_1} m_1$.
  Let $d$ be their strict simple majority outcome.
  For any $\{u, v\} \in \binom{X}{2}$, if $\{u, v\} \not \subset \{x_1, \ldots, x_k\}$
  then $u$ ties $v$ among the $c_i$ and the $c'_i$, so $\{u, v\} \notin \dom d$.
  Any two points on the cycle $(x_1, \ldots, x_{k-1})$ which are not adjacent are tied
  among the $c_i$ and among the $c'_i$.  Any point on the cycle besides $x_1$
  or $x_{k-1}$ is also tied with $x_k$ among the $c_i$ and among the $c'_i$.
  Any pair of consecutive points $(x_i, x_j)$ on the cycle, besides $(x_{k-1}, x_1)$,
  are tied among the $c'_i$, but the majority of the $c_i$ pick $x_i$.
  So $d\{x_i, x_j\} = x_i$.
  $x_{k-1}$ and $x_k$ tie among the $c_i$, but the
  majority of the $c'_i$ pick $x_{k-1}$, so $d\{x_{k-1},x_k\} = x_{k-1}$.
  Similarly $d\{x_k, x_1\} = x_k$.
  The only remaining pair to consider is $\{x_1, x_{k - 1}\}$.
  $d\{x_1, x_{k-1}\}$ is defined iff
  \[
  \sum_{i=1}^{m} W^{x_1}_{x_{k-1}}(d_i) \neq 0.
  \]
  The left hand side is equal to
  \begin{align*}
  \frac{L}{l_0} \sum_{i=1}^{m_0} W^{x_1}_{x_{k-1}}(c_i) +
  \frac{L}{l_1} \sum_{i=1}^{m_1} W^{x_1}_{x_{k-1}}(c'_i) \\
  = \frac{L}{l_0} (-l_0) + \frac{L}{l_1} (l_1) = 0.
  \end{align*}
  Thus $\{x_1, x_{k-1}\} \notin \dom d$, and $d\{x_i,x_j\} = x_i$ iff $j \equiv
  i + 1$ mod $k$, so $d = c^{x_1, \ldots, x_k} \in \majcl(\subC)$.

  By induction on $k$, any cyclic choice function on $X$
  is in $\majcl(\subC)$.
\end{proof}

\begin{clm}
  \label{pseudinmaj}
  If $\subC \subset \C$ is symmetric, balanced, and nontrivial,
  then every pseudo-balanced choice function $d \in \C$ is in $\majcl(\subC)$.
\end{clm}
\begin{proof}
  Suppose $d$ is an arbitrary pseudo-balanced choice function.
  Every edge of $\tor(d)$ is on a directed cycle, so consider the
  decomposition of $\tor(d)$ into cycles $C_{x,y}$ for each edge $x \to y$ in
  $\tor(d)$.  Let $d_{x,y}$ be the cyclic function corresponding to $C_{x,y}$ for
  each edge $x \to y$ in $\tor(d)$.  By Claim~\ref{triclm}, $\subC$ meets the
  requirements of Claim~\ref{cycclm}, so every cyclic function on $X$ is in
  $\majcl(\subC)$; in particular, each $d_{x,y} \in \majcl(\subC)$.  If we
  combine all the finite sets of voters which yielded each $d_{x,y}$, we get
  another finite set of voters from $\subC$, since there are finitely many edges
  in $\tor(d)$.

  For any $\{u,v\} \in \binom{X}{2} \setminus \dom d$,
  the edge $u \to v$ is not in any cycle $C_{x,y}$.
  $\sum W^u_v(c_i) = 0$ over the $c_i$ yielding any $d_{x,y}$, so
  $\sum W^u_v(c_i) = 0$ over all our voters and $u$ and $v$ tie in their strict
  simple majority outcome.

  For any $\{u,v\} \in \binom{X}{2}$ with $d\{u,v\} = u$, the edge $u \to v$ is in
  $\tor(d)$.  Since each cycle $C_{x,y}$ has edges only from $\tor(d)$, each
  population of $c_i$ yielding some $d_{x,y}$ has $\sum W^u_v(c_i) \geq 0$, and
  in particular the population yielding $d_{u,v}$ has $\sum W^u_v(c_i) > 0$, so
  over all our voters $\sum W^u_v(c_i) > 0$.

  Thus the strict simple majority outcome is $d$, and $d \in \majcl(\subC)$.
\end{proof}
Combining Claims~\ref{majinpseud} and \ref{pseudinmaj}, we have that
$\majcl(\subC)$ is the set of all pseudo-balanced choice functions
whenever $\subC$ is symmetric, balanced, and nontrivial.

\newpage
\section{Chaotic Choices}
\label{elsesec}
\begin{dfn}
  For a choice function $c \in \C$,
  \begin{list}{(\alph{dlcount})}{\usecounter{dlcount}}
    \item
      For $x \in X$, the \emph{valence} of $x$ with $c$ is
      \[
      \val_c(x) = \sum_{y \in X} W^x_y(c).
      \]
    \item
      For $\ell \in \{-1, 0, 1\}$, let
      \[
      V_\ell(c) = \{(\val_c(x) - \ell, \val_c(y) + \ell)) : \{x, y\} \in
      \binom{X}{2}, \, W^x_y(c) = \ell\}.
      \]
      These are the valence pairs, with winners on the right in $V_{-1}$,
      ties in $V_0$, winners on the left in $V_1$, and with 1 subtracted from the
      valence of the winner.  In fact,
      \[
      V_\ell(c) = \left\{\sum_{z \in X \setminus \{x,y\}}
       \left(W^x_z(c), W^y_z(c)\right) : \{x,y\} \in \binom{X}{2}, \, W^x_y(c) = \ell\right\}.
      \]
    \item
      For a subset $A$ of $\Q \times \Q$, let $\conv(A)$ be the convex hull of $A$ in $\Q \times \Q$.
    \item
      Let
      \begin{align*}
        V^*(c)& = \{a \bar{v}_1 + (1 - a) \bar{v}_0 : \bar{v}_1 \in \cvw,
        \bar{v}_0 \in \cvt, a \in (0, 1]_\Q \}.
      \end{align*}
      These are the convex hulls of the ``winning'' valence pairs, together with ties,
      requiring some contribution from a non-tied pair.
  \end{list}
\end{dfn}

\begin{clm}
  \label{reflectclm}
  $(k_0,k_1) \in V_\ell(c) \iff (k_1,k_0) \in V_{-\ell}(c)$, for any
  $\ell \in \{-1, 0, 1\}$.
\end{clm}
\begin{proof}
  A pair $\left(\val_c(u) + 1, \val_c(v) - 1\right)$ is in $V_{-1}(c)$ iff
  $W^u_v(c) = -1 \iff W^v_u(c) = 1$, so $(\val_c(v) - 1, \val_c(u) + 1) \in V_1(c)$.

  A pair $\left(\val_c(u), \val_c(v)\right)$ is in $V_0(c)$
  iff $\{u, v\} = \{v, u\} \notin \dom c$, so
  $\left(\val_c(v), \val_c(u)\right) \in V_0(c)$.
\end{proof}

\begin{clm}
  \label{partifflines}
  An imbalanced $c \in \C$ is partisan iff
  \begin{itemize}
    \item
      $V_1(c)$ lies on a line parallel to the line $y = x$, and
    \item
      $V_0(c)$ is contained in the line $y = x$.
  \end{itemize}
\end{clm}
\begin{proof}
  Suppose that $V_0(c)$ is all on $y = x$.  Thus, candidates are only
  ever tied with others of the same valence.
  Suppose further that $V_1(c)$ is on a line $y = x - b$.
  Then whenever $c\{w,v\} = w$, $\val_c(w) - \val_c(v) = b + 2$ is constant;
  the valence of any winner is always $b + 2$ more than the defeated.
  Since $c$ is imbalanced, there is such a pair $\{w,v\}$.
  If any candidate $z$ has a
  valence different from that of $w$, it cannot be tied with $w$, so $\{w,z\} \in \dom c$.
  Therefore the valences differ by $b+2$, but if $\val_c(z) \neq \val_c(v)$,
  then $z$ can neither tie nor be comparable to $v$, a contradiction.
  So every candidate has the valence of $w$ or the valence of $v$.
  If two candidates with the same valence do not tie,
  then $b = -2$ and $\val_c(v) = \val_c(w)$.
  In this case, every candidate has the same
  valence, which would mean that $c$ is balanced; a contradiction.
  Furthermore, every candidate with the high valence
  must defeat everyone of the low valence, since they cannot be tied.
  This is the definition of a partisan function.

  Conversely, suppose that $c$ is partisan.  Then two elements are tied only
  if they are both in the winning subset, or both in the losing subset;
  in either case, they have the same valence,
  so $V_0(c)$ is contained in the line $y = x$.
  $V_1(c)$ is the single point $(\val_c(w) - 1, \val_c(v) + 1)$
  for any $w$ in the winning subset and $v$ in the losing subset.
  Naturally, this point is contained in a line parallel to $y = x$.
\end{proof}

\begin{clm}
  \label{abovebelow}
  If $c \in \C$ is imbalanced, then a point of $V^*(c)$ lies above
  the line $y = -x$, and a point of $V^*(c)$ lies below it.
\end{clm}
\begin{proof}
  Let $v_1$ and $v_2$ be two candidates in $X$ having the highest valences with $c$.
  Since $c$ is imbalanced, $\val_c(v_1) > 0$.
  The sum of the corresponding pair in any $V_\ell(c)$ is $\val_c(v_1) + \val_c(v_2)$.
  If this is less than or equal to $0$, then $\val_c(v_2) \leq -1$.
  Thus all other valences are at most $-1$, so the average valence
  is strictly smaller than the average of $\val_c(v_1)$ and $\val_c(v_2)$, at most
  $0$.  This is a contradiction, since the average valence is always $0$.
  Therefore, the sum of the valence pair for $v_1$ and $v_2$ in any $V_\ell(c)$
  satisfies $y + x > 0$.
  If the pair is in $V_1(c)$, it is in $V^*(c)$.
  If it is in $V_{-1}(c)$, then its reflection in $V_1(c)$ (hence in $V^*(c)$) still
  satisfies $x + y > 0$.
  If it is in $V_0(c)$, then points arbitrarily close to the pair, along any line
  to a point of $V_1(c)$, are in $V^*(c)$.  Some of these points are above $y = -x$.

  Let $u_1$ and $u_2$ be two candidates in $X$ having the smallest valences with $c$.
  By imbalance, $\val_c(u_1) < 0$.
  Suppose $\val_c(u_1) + \val_c(u_2) \geq 0$.
  Then the average of these two valences is at least $0$, and
  the average over all candidates is larger; it is strictly larger, since
  $\val_c(v_1) > 0$ figures in the average.
  This is a contradiction,
  so the sum of the valence pair for $u_1$ and $u_2$ in any $V_\ell(c)$
  satisfies $y + x < 0$.
  If this pair is in $V_{-1}(c)$, then its reflection in $V_1(c)$ still satisfies
  $x + y < 0$.  Otherwise it is in $V^*(c)$, or arbitrarily close points are
  in $V^*(c)$.
\end{proof}

To establish that every choice function is in the majority closure of a
chaotic symmetric set, we shall first establish that
some $c$ therein satisfies a rather abstruse condition which we'll call
``valence-imbalance.''
\begin{dfn}
  A choice function $c \in \C$ is \emph{valence-imbalanced} iff
  $\ntnt$ can be represented as
  $r_{-1} \bar{v}_{-1} + r_0 \bar{v}_0 + r_1 \bar{v}_1$ where
  \begin{enumerate}
    \item
      Each $\bar{v}_\ell$ is a pair in $\cvl$,
    \item
      $r_{-1}, r_0, r_1 \in [0,1]_\Q$,
    \item
      $r_{-1} + r_0 + r_1 = 1$,
    \item
      $r_{-1} \neq r_1$.
  \end{enumerate}
\end{dfn}
Note that $\ntnt \in V^*(c)$ implies that $c$ is valence-imbalanced.

\begin{clm}
  \label{parlineclm}
  If $\bar{u}$ is strictly between two points of $\cvw$ on a line segment not
  parallel to $y = x$, then the nearest point on $y = x$ to $\bar{u}$ is of the
  form $r_{-1}\bar{v}_{-1} + r_1 \bar{v}_1$, where $\bar{v}_\ell \in \cvl$,
  $r_\ell \in [0,1]_\Q$, $r_{-1} + r_1 = 1$, but $r_{-1} \neq r_1$.
\end{clm}
\begin{proof}
  Say $\bar{u} = (a, a + b)$.  If $b = 0$, the assertion holds.
  Otherwise, let $\bar{w} = \left(a + \frac{b}{2}, a + \frac{b}{2}\right)$;
  $\bar{w}$ is the nearest point to $\bar{u}$ on $y = x$.

  There are $\bar{p}_0 = (a_0, b_0)$ and $\bar{p}_1 = (a_1, b_1)$ in $\cvw$ so
  that $\bar{u}$ is strictly between $\bar{p}_0$ and $\bar{p}_1$.  If the line
  segment $\Lambda$ between $\bar{p}_0$ and $\bar{p}_1$ is perpendicular to $y =
  x$, then either $\bar{w}$ is in $\Lambda$, and the assertion holds, or
  $\bar{w}$ is between $\bar{p}_1$ and the reflection of $\bar{p}_0$, $\bar{q}_0
  = (b_0, a_0)$.  So for some $r_{-1}$ and $r_{1}$, $r_{-1} \bar{q}_0 + r_1 \bar{p}_1 =
  \bar{w} = \frac{1}{2} \bar{q}_0 + \frac{1}{2} \bar{p}_0$.  If $r_{-1} = r_1 =
  \frac{1}{2}$, then $\bar{p}_1 = \bar{p}_0$, contradicting that $\bar{u}$ is
  strictly between them.  So $r_{-1} \neq r_1$.

  If $\Lambda$ is not perpendicular to the line $y = x$, then we may choose the
  points $\bar{p}_0$ and $\bar{p}_1$ so that $a_0 + b_0 < 2a + b < a_1 + b_1$,
  and on the same side of $y = x$.  The reflection of $\Lambda$ over $y = x$ lies
  on a line, say $L$.  $L$ contains $\bar{v} = (a+b, a)$, $\bar{q}_0 = (b_0,
  a_0)$, and $\bar{q}_1 = (b_1, a_1)$, all in $\cvn$; so $\bar{v}$ lies strictly
  between $\bar{q}_0$ and $\bar{q}_1$ within $\cvn$.

  \begin{figure}[t]
    \begin{tikzpicture}[scale=0.45,minimum size=3pt]
  \draw[<->, thin] (2,2) -- node[right,very near end] {$y = x$} (9.5,9.5);
  \coordinate (l0) at (1,4.5);
  \coordinate (l1) at (4.5,1);
  \node (p0) at ( 3.7,2)   [fill,circle,inner sep=0pt,label=below:$\bar{p}_0$] {};
  \node (p1) at (10,  4)   [fill,circle,inner sep=0pt,label=right:$\bar{p}_1$] {};
  \node (q0) at ( 2,  3.7) [fill,circle,inner sep=0pt,label=left:$\bar{q}_0$] {};
  \node (q1) at ( 4, 10)   [fill,circle,inner sep=0pt,label=above:$\bar{q}_1$] {};

  \draw[dashed, very thin] (l0) -- +(6, 6) node[name=t0] {};
  \draw[dashed, very thin] (l1) -- +(6, 6) node[name=t1] {};
  \node (u)  at (intersection of p0--p1 and l1--t1) [fill,circle,inner sep=0pt,label=below right:$\bar{u}$] {};
  \node (v)  at (intersection of q0--q1 and l0--t0) [fill,circle,inner sep=0pt,label=above left:$\bar{v}$] {};
  \draw[dotted] (p0) -- (p1);
  \draw[dotted] (q0) -- (q1);
  \draw[dotted] (u) -- (v);
  \node (w)  at (intersection of u--v and 0,0--1,1) [fill,circle,inner sep=0pt,label=above:$\bar{w}$] {};
  \node (uN) at (intersection of 9,0--9,1 and u--p1) [fill,circle,inner sep=0pt,label=below:$\bar{u}_N$] {};
  \draw (uN) -- (intersection of uN--w and q0--q1) node [fill,circle,inner sep=0pt,label=below right:$\bar{v}_N$] {};
\end{tikzpicture}
    \caption{The situation of Claim~\ref{parlineclm}.  $\bar{u}_N$ and $\bar{v}_N$
    cannot be equidistant from $\bar{w}$, since one lies within the dashed parallel
    lines, and one lies without.}
    \label{parfig}
  \end{figure}
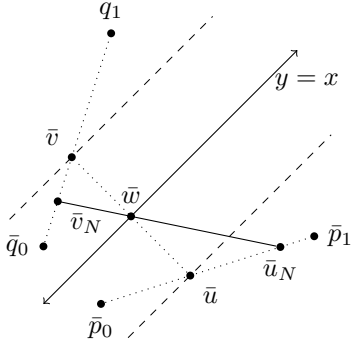

  Consider the sequence
  \[
  \bar{u}_n = \frac{1}{n} \bar{p}_1 + \left(1 - \frac{1}{n}\right) \bar{u},
  \]
  for $n \geq 1$, approaching $\bar{u}$ as $n \to \infty$.  For each $n$, let
  $\bar{v}_n$ be the intersection of the line through $\bar{u}_n$ and $\bar{w}$
  with the line $L$. (If necessary, define the $\bar{v}_n$ only for $n > j$,
  if the line from $\bar{u}_j$ through $\bar{w}$ is parallel to $L$; at most one
  such $j$ can exist.)
  As $n \to \infty$, $\bar{v}_n$ approaches $\bar{v}$.  So for
  some $N \in \N$ and all $n \geq N$, $\bar{v}_n$ is in the interval
  $\left(\bar{q}_0, \bar{q}_1\right)$, an open neighborhood of $\bar{v}$ in $L$.
  So
  \[
  \bar{w} = r_{-1} \bar{u}_N + r_1 \bar{v}_N,
  \]
  for some $r_{-1}$ and $r_1$ in $[0,1]_\Q$ with $r_{-1} + r_1 = 1$.

  If $r_{-1} = r_1 = \frac{1}{2}$, then $\bar{v}_N = \bar{w} + (\bar{w} -
  \bar{u}_N)$ is just as far from the line $y = x$ as $\bar{u}_N = \bar{w} -
  (\bar{w} - \bar{u}_N)$.  But notice, either
  \begin{itemize}
    \item
      both the interval $(\bar{u}, \bar{p}_1)$ on $\Lambda$ and the
      interval $(\bar{v}, \bar{q}_1)$ on $L$ are farther from $y = x$ than $\bar{u}$
      is, while the intervals $(\bar{p}_0, \bar{u})$ and $(\bar{q}_0, \bar{v})$ are
      closer, or
    \item
      the intervals $(\bar{u}, \bar{p}_1)$ and $(\bar{v}, \bar{q}_1)$ are closer to $y
      = x$ than $\bar{u}$ is, while the intervals $(\bar{p}_0, \bar{u})$ and
      $(\bar{q}_0, \bar{v})$ are farther.
  \end{itemize}
  $\bar{u}_N$ is in the half-plane $y + x > 2 a + b$, and $\bar{w}$ is the
  sole intersection of the line $\{\bar{u}_N + t(\bar{w} - \bar{u}_N) : t \in \Q\}$
  with the line $y + x = 2 a + b$.  Thus $\bar{v}_N$ is in the half-plane $y + x < 2 a +
  b$, hence in the interval $(\bar{q}_0, \bar{v})$ of $L$.  But then the
  distance from $y = x$ to $\bar{u}$ is strictly between the distances from $y = x$
  to $\bar{v}_N$ and to $\bar{u}_N$, contradicting that these distances are
  equal.

  So $r_{-1} \neq r_1$.
\end{proof}

\begin{clm}
  \label{imbalclm}
  If $c$ is chaotic, then $c$ is valence-imbalanced.
\end{clm}
\begin{proof}
  We consider five cases.

  \begin{case}
    \label{fourcase1}
    $V_0(c) \setminus \{\ntnt\}$ is not contained in $y = x$, nor in one open
    half-plane of $y = -x$.
  \end{case}

  By Claim~\ref{abovebelow}, there is a point $\bar{v}$ of $V^*(c)$ on the line $y
  = -x$.  There is a point $(k_0, -k_0) \in \cvt$, $k_0 \neq 0$.  Either $\ntnt$
  is between $\bar{v}$ and $(k_0, -k_0)$ or it is between $\bar{v}$ and $(-k_0,
  k_0)$; either way, it is in $V^*(c)$, so $c$ is valence-imbalanced.

  \begin{case}
    $V_0(c) \setminus \{\ntnt\}$ is not contained in $y = x$, but is contained in
    one open half-plane of $y = -x$.
  \end{case}

  There is $(k_0, k_1) \in V_0(c)$ such that $k_0 \neq k_1$ and $k_0 \neq -k_1$.
  By Claim~\ref{abovebelow}, there is a point $\bar{v} \in V_1(c)$ strictly on
  the other side of $y = -x$.  If $\bar{v}$ is on $y = x$, then $\ntnt$ is on the
  line segment between $\bar{v}$ and $\bar{w} = \left(\frac{k_0 + k_1}{2}, \frac{k_0 +
  k_1}{2}\right)$, hence in $V^*(c)$, so $c$ is valence-imbalanced.
  If $\bar{v}$ is not on $y = x$, then it has a reflection $\bar{u} \in V_{-1}(c)$.
  $\conv\{\bar{v}, \bar{u}, \bar{w}\}$ contains an open disc about $\ntnt$.
  Since there are points in $\cvt$ arbitrarily close to $\bar{w}$ on the line
  between $(k_0, k_1)$ and $(k_1, k_0)$, we may choose
  $(k'_0, k'_1)$ so $k'_0 \neq k'_1$ and $\ntnt \in \conv\{\bar{v}, \bar{u}, (k'_0, k'_1)\}$.

  Suppose $c$ is valence-balanced.  Say $\bar{v}$ is $(v_0, v_1)$.  Then
  \[
  \ntnt = r \bar{v} + (1 - 2r)(k'_0, k'_1) + r \bar{u}
  = r(v_0 + v_1, v_1 + v_0) + (1 - 2r)(k'_0, k'_1),
  \]
  so $k'_0 = \frac{-r (v_0 + v_1)}{1 - 2r} = k'_1$, a contradiction.
  Therefore $c$ is valence-imbalanced.

  \begin{case}
    $V_0(c)$ is contained in $y = x$, and $V_1(c)$ is contained in a line
    parallel to $y = x$.
  \end{case}

  Then $c$ is partisan by Claim~\ref{partifflines}, contradicting that $c$ is chaotic.

  \begin{case}
    \label{threecase}
    $V_0(c)$ is contained in $y = x$, $\cvw$ contains a line segment not parallel to
    $y = x$, and $V_1(c)$ has points on either side of $y = -x$.
  \end{case}

  Then there is a point of $y = -x$ strictly between two points of $\cvw$ such that the
  line segment between them is not parallel to $y = x$.
  By Claim~\ref{parlineclm}, $(0,0) = r_{-1}\bar{v}_{-1} + r_1\bar{v}_1$ where
  $\bar{v}_\ell \in \cvl$, $r_\ell \in [0,1]_\Q$, $r_{-1} + r_1 = 1$, but $r_{-1}
  \neq r_1$, i.e. $c$ is valence-imbalenced.

  \begin{case}
    \label{fourcase2}
    $V_0(c)$ is contained in $y = x$, $\cvw$ contains a line segment not parallel to
    $y = x$, and $V_1(c)$ is entirely on one side of $y = -x$.
  \end{case}

  Let $\bar{v}$ be strictly between two points of $V_1(c)$ on a line segment not
  parallel to $y = x$.
  By Claim~\ref{parlineclm}, the nearest point $\bar{w}$ to $\bar{v}$ on the line
  $y = x$ is $r_{-1}\bar{v}_{-1} + r_1\bar{v}_1$ where
  $\bar{v}_\ell \in \cvl$, $r_\ell \in [0,1]_\Q$, $r_{-1} + r_1 = 1$, but $r_{-1}
  \neq r_1$.
  By Claim~\ref{abovebelow}, $V_0(c)$ has a point $\bar{v}_0 = (k, k)$ strictly on
  the other side of $y = -x$ from $\bar{v}$, and hence from $\bar{w}$.  So for some
  $r \in [0,1]_\Q$,
  \[
  \ntnt = r \bar{v}_0 + (1 - r)(r_{-1}\bar{v}_{-1} + r_1\bar{v}_1) \\
  = (1 - r) r_{-1} \bar{v}_{-1} + r \bar{v}_0 + (1 - r) r_1 \bar{v}_1,
  \]
  and $(1 - r) r_{-1} \neq (1 - r) r_1$, so $c$ is valence-imbalanced.
\end{proof}

\begin{clm}
  \label{arrowclm}
  If $\subC \subseteq \C$ is symmetric and $c \in \subC$ is valence-imbalanced,
  then for each $\{x, y\} \in \binom{X}{2}$ there is $d \in \majcl(\subC)$ such that
  $\dom d = \{\{x,y\}\}$.
  \end{clm}
\begin{proof}
  Suppose $\{x, y\} \in \binom{X}{2}$.
  There are $\bar{v}_\ell \in \cvl$ and
  $r_\ell \in [0,1]_\Q$, $r_{-1} + r_0 + r_1 = 1$, $r_{-1} \neq r_1$, so
  $r_{-1} \bar{v}_{-1} + r_0 \bar{v}_0 + r_1 \bar{v}_1 = \ntnt$.
  \[
  \bar{v}_{-1} = \sum_{\bar{v} \in V_{-1}(c)} s^{-1}_{\bar{v}} \bar{v}, \quad
  \bar{v}_0 = \sum_{\bar{v} \in V_0(c)} s^0_{\bar{v}} \bar{v}, \quad
  \bar{v}_1 = \sum_{\bar{v} \in V_1(c)} s^1_{\bar{v}} \bar{v},
  \]
  where $\sum_{\bar{v} \in V_\ell(c)} s^\ell_{\bar{v}} = 1$ for each $\ell$.
  Each $\bar{v} \in V_\ell(c)$ is $(\val_c(u^\ell_{\bar{v}}) - \ell,
  \val_c(w^\ell_{\bar{v}}) + \ell)$ for some $\{u^\ell_{\bar{v}},
  w^\ell_{\bar{v}}\} \in \binom{X}{2}$.  Let $\sigma^\ell_{\bar{v}}$ be a
  permutation taking $u^\ell_{\bar{v}}$ to $x$ and $w^\ell_{\bar{v}}$ to $y$, i.e.
  \[
  \sigma^\ell_{\bar{v}} =
  \begin{cases}
    (x, y) & \text{if } u^\ell_{\bar{v}} = y, w^\ell_{\bar{v}} = x, \\
    (x, w^\ell_{\bar{v}}, y) & \text{if } u^\ell_{\bar{v}} = y, w^\ell_{\bar{v}}
    \neq x, \\
    (u^\ell_{\bar{v}}, x)(w^\ell_{\bar{v}}, y) & \text{otherwise.}
  \end{cases}
  \]

  Let $\Gamma_{x,y}$ be all the permutations of $X$ which fix $x$ and $y$, and let
  \[
  \prt = \frac{1}{\abs{\Gamma_{x,y}}} \sum_{\ell \in \{-1, 0, 1\}} r_\ell
  \sum_{\bar{v} \in V_\ell(c)} s^\ell_{\bar{v}} \sum_{\tau \in \Gamma_{x,y}}
  \pr\left(c^{\tau \sigma^\ell_{\bar{v}}}\right).
  \]
  Of course, the sum of the coefficients is
  \[
  \sum_{\substack{
    \ell \in \{-1, 0, 1\}\\
    \bar{v} \in V_\ell(c)\\
    \tau \in \Gamma_{x,y}}}
  \frac{r_\ell s^\ell_{\bar{v}}}{\abs{\Gamma_{x,y}}} \  = \
  \frac{\abs{\Gamma_{x,y}}}{\abs{\Gamma_{x,y}}} \sum_{\ell \in \{-1, 0, 1\}}
  r_\ell \  = \  1.
  \]
  Since $r_\ell$, $s^\ell_{\bar{v}}$, and $\frac{1}{\abs{\Gamma_{x,y}}}$ are in $[0,1]_\Q$,
  $\frac{r_\ell s^\ell_{\bar{v}}}{\abs{\Gamma_{x,y}}} \in [0,1]_\Q$.
  So $\prt \in \prcl(\subC)$.

  Consider $t_{x,y}$.  Each $\tau$ fixes $x$ and $y$, so
  \begin{align*}
    c^{\tau \sigma^\ell_{\bar{v}}}\{x,y\}& = x \iff \\
    c^{\sigma^\ell_{\bar{v}}}\{x,y\}& = x \iff \\
    c\{u^\ell_{\bar{v}},w^\ell_{\bar{v}}\}& = u^\ell_{\bar{v}} \iff \ell = 1.
  \end{align*}
  So $W^x_y(c^{\tau \sigma^\ell_{\bar{v}}}) = \ell$.
  Thus
  \begin{align*}
    t_{x,y}& = \frac{1}{\abs{\Gamma_{x,y}}}
    \sum_{\ell \in \{-1, 0, 1\}} r_\ell
    \sum_{\bar{v} \in V_\ell(c)} s^\ell_{\bar{v}}
    \sum_{\tau \in \Gamma_{x,y}} \ell \\
    & = \frac{\abs{\Gamma_{x,y}}}{\abs{\Gamma_{x,y}}}
    \sum_{\ell \in \{-1, 0, 1\}} \ell \, r_\ell \  = \  r_1 - r_{-1}.
  \end{align*}
  Since $r_{-1} \neq r_1$, $t_{x,y} \neq 0$.

  Consider $t_{x,u}$ where $u \notin \{x,y\}$.
  \begin{align*}
    c^{\tau \sigma^\ell_{\bar{v}}}\{x,u\}& = x \iff \\
    c^{\sigma^\ell_{\bar{v}}}\{x,\tau^{-1}(u)\}& = x \iff \\
    c\{u^\ell_{\bar{v}},u^*\}& = u^\ell_{\bar{v}},
  \end{align*}
  where
  \[
  u^* = {\left(\sigma^\ell_{\bar{v}}\right)}^{-1}(\tau^{-1}(u)) =
  \begin{cases}
    \tau^{-1}(u) & \text{if } \tau^{-1}(u) \notin \{u^\ell_{\bar{v}}, w^\ell_{\bar{v}}\}, \\
    x & \text{if } \tau^{-1}(u) = u^\ell_{\bar{v}}, \\
    y & \text{if } \tau^{-1}(u) = w^\ell_{\bar{v}}.
  \end{cases}
  \]
  $u^*$ is never $u^\ell_{\bar{v}}$ or $w^\ell_{\bar{v}}$.
  For a given $\ell$ and $\bar{v}$, $u^*$ varies equally over all elements of
  $X$ besides $\{u^\ell_{\bar{v}}, w^\ell_{\bar{v}}\}$, so
  \begin{multline*}
    \sum_{\tau \in \Gamma_{x,y}} W^x_u(c^{\tau \sigma^\ell_{\bar{v}}})
    = \sum_{\tau \in \Gamma_{x,y}} W^{u^\ell_{\bar{v}}}_{u^*}(c)
    = \sum_{\tau \in \Gamma_{x,y,u}}
    \sum_{u^* \in X \setminus \{u^\ell_{\bar{v}}, w^\ell_{\bar{v}}\}} W^{u^\ell_{\bar{v}}}_{u^*}(c) \\
    = \abs{\Gamma_{x,y,u}} \left( \val_c(u^\ell_{\bar{v}}) -
      W^{u^\ell_{\bar{v}}}_{u^\ell_{\bar{v}}}(c) - W^{u^\ell_{\bar{v}}}_{w^\ell_{\bar{v}}}(c)\right)
    = \abs{\Gamma_{x,y,u}} \left( \val_c(u^\ell_{\bar{v}}) - \ell \right).
  \end{multline*}
  This is because, for each $u^* \in X \setminus \{u^\ell_{\bar{v}},
  w^\ell_{\bar{v}}\}$, there are $\abs{\Gamma_{x,y,u}}$
  permutations $\tau$ which take $u$ to $u^*$.
  $W^u_u(c) = 0$ for any $u$, and by choice of $u^\ell_{\bar{v}}$ and $w^\ell_{\bar{v}}$,
  $W^{u^\ell_{\bar{v}}}_{w^\ell_{\bar{v}}}(c) = \ell$.

  Recall that the $u^\ell_{\bar{v}}$ were chosen so that $\val_c(u^\ell_{\bar{v}}) -
  \ell$ is the first coordinate of $\bar{v}$.  Recall also that the $r_\ell$,
  $s^\ell_{\bar{v}}$, and valence pairs $\bar{v}$ were chosen so
  \[
  \sum_{\ell \in \{-1, 0, 1\}} r_\ell
  \sum_{\bar{v} \in V_\ell(c)} s^\ell_{\bar{v}} \bar{v} = \ntnt.
  \]
  \begin{align*}
    t_{x,u}& = \frac{1}{\abs{\Gamma_{x,y}}}
    \sum_{\ell \in \{-1, 0, 1\}} r_\ell
    \sum_{\bar{v} \in V_\ell(c)} s^\ell_{\bar{v}}
    \sum_{\tau \in \Gamma_{x,y}} W^x_u\left( c^{\tau \sigma^\ell_{\bar{v}}} \right) \\
    & = \frac{1}{\abs{\Gamma_{x,y}}}
    \sum_{\ell \in \{-1, 0, 1\}} r_\ell
    \sum_{\bar{v} \in V_\ell(c)} s^\ell_{\bar{v}}
    \abs{\Gamma_{x,y,u}} \left( \val_c(u^\ell_{\bar{v}}) - \ell \right) \\
    & = 0.
  \end{align*}
  Similarly, $t_{y,u} = 0$ for all $u \notin \{x,y\}$.

  Consider $t_{u,w}$ with neither $u$ nor $w$  in $\{x,y\}$.  Then
  \begin{align*}
    c^{\tau \sigma^\ell_{\bar{v}}}\{u,w\}& = u \iff \\
    c^{\sigma^\ell_{\bar{v}}}\{\tau^{-1}(u),\tau^{-1}(w)\}& = \tau^{-1}(u) \iff \\
    c\{u^*,w^*\}& = u^*,
  \end{align*}
  where $u^* = {(\sigma^\ell_{\bar{v}})}^{-1}(\tau^{-1}(u))$ and
  $w^* = {(\sigma^\ell_{\bar{v}})}^{-1}(\tau^{-1}(w))$.
  But $\tau^{-1}(u)$ and $\tau^{-1}(w)$ are never $x$ or $y$,
  so $u^*$ is never $u^\ell_{\bar{v}}$ and $w^*$ is never $w^\ell_{\bar{v}}$.
  As $\tau$ varies, $u^*$ and $w^*$ vary equally over all members of $X$
  except $u^\ell_{\bar{v}}$ and $w^\ell_{\bar{v}}$,
  so $c\{u^*, w^*\} = u^*$ just as often as $c\{u^*, w^*\} = w^*$.
  Thus
  \[
  \sum_{\tau \in \Gamma_{x,y}} W^u_w(c^{\tau \sigma^\ell_{\bar{v}}}) = 0.
  \]
  Hence
  \begin{align*}
    t_{u,v}& =
    \frac{1}{\abs{\Gamma_{x,y}}}
    \sum_{\ell \in \{-1, 0, 1\}} r_\ell
    \sum_{\bar{v} \in V_\ell(c)} s^\ell_{\bar{v}}
    \sum_{\tau \in \Gamma_{x,y}} W^u_w(c^{\tau \sigma^\ell_{\bar{v}}})
    = 0.
  \end{align*}
  So $d = \maj(\prt) \in \majcl(\subC)$ has $\dom d = \{\{x,y\}\}$;
  all other pairs have $t_{u,v} = 0$.
\end{proof}

\begin{clm}
  If $\subC \subseteq \C$ is symmetric and chaotic, then $\majcl(\subC) = \C$.
\end{clm}
\begin{proof}
  Suppose $f \in \C$.

  By Claim~\ref{imbalclm}, there is some $c \in \subC$ which is valence-imbalanced.
  So by Claim~\ref{arrowclm}, for any $\{x, y\} \in \binom{X}{2}$ there is $d \in \majcl(\subC)$
  with $\dom d = \{\{x,y\}\}$.  If $d\{x,y\} = y$, by symmetry there is $d' \in \majcl(\subC)$
  with domain $\{\{x, y\}\}$ and $d'\{x,y\} = x$.
  So for every $\{x,y\} \in \dom f$, let $d_{x,y} \in \majcl(\subC)$ such that
  $\dom d_{x,y} = \{\{x,y\}\}$ and $d_{x,y}\{x,y\} = f\{x,y\}$.
  Now combine the voter populations which yielded each $d_{x,y}$.

  $f$ was arbitrary, so we have shown $\C \subseteq \majcl(\subC)$.
  Hence $\majcl(\subC) = \C$.
\end{proof}

\newpage
\section{Bounds}
We shall consider the loose upper bounds, implied by the preceding proofs,
on the necessary number of voters from a
symmetric set to yield an arbitrary function in its majority closure,
in terms of $\n$.
There is no doubt that much tighter bounds could be obtained.

\subsection*{Trivial}
0 functions suffice to obtain no outcome.

\subsection*{Balanced}
In order to create a triangular function from
a chosen function $c$,
$(\n - 3)!$ permutations of the rest of the set were used.
These were repeated $3$ times, to establish each pair of edges in the triangle.

An arbitrary pseudo-balanced function has each edge in a directed cycle,
so one cycle per edge suffices;
so we have at most $\binom{\n}{2} = \frac{\n(\n - 1)}{2}$ cycles.  Each is
constructed with two fewer triangles than the number of nodes it contains.
This is at most $\n - 2$, if some cycle contains every node.
So all together we have
\[
\frac{1}{2} \n (\n - 1) (\n - 2) (\n - 3)!\, 3 = \frac{3 \n!}{2}.
\]

\subsection*{Partisan}
There are at most $\n$ tiers in an arbitrary tiered function,
each with at most $\n$ elements (of course, no function meets both
conditions.) For each such element $x$, we take functions for each permutation
holding $x$ fixed; there are $(\n - 1)!$ such.
So we have $\n\, \n (\n - 1)! = \n\, \n!$ as an upper bound.

\subsection*{Imbalanced, nonpartisan, but not chaotic}
The proof of Claim~\ref{triclm} gives the size of the population $T_z$ yielding $c^{a,b,z}$
as $\abs{T_z} = 3 (\n - 3)!$, with $m = 2 (\n - 3)!$ votes for the winner of each pair.
We chose a partisan function $c$ with $l$ candidates in the winning tier.
For each of at most $\binom{\n}{2}$ edges $a \to b$ and each of the $(\n -2)$
$z \in X \setminus \{a,b\}$, we used $k$ copies of $T_z$, where
\[
k =
\begin{cases}
  (\n - 2) \binom{\n  -3}{l - 1} & \text{if } l \leq \frac{\n}{2}, \\
  (\n - 2) \binom{\n - 3}{l - 2} & \text{if } l > \frac{\n}{2},
\end{cases}
\]
and $m$ copies of a population of partisan functions $D$, where
\[
\abs{D} =
\begin{cases}
  (\n - 2l) \binom{\n - 1}{l} + (\n - 2) \binom{\n - 2}{l - 1} & \text{if } l \leq \frac{\n}{2}, \\
  (2l - \n) \binom{\n - 1}{l - 1} + (\n - 2) \binom{\n - 2}{l - 1} & \text{if } l > \frac{\n}{2}.
\end{cases}
\]
In fact, $k' = \frac{k}{\gcd(m,k)}$ copies of $T_z$, and $m' = \frac{m}{\gcd(m,k)}$ copies
of D, would suffice.  Certainly
$\binom{\n - 3}{l - 1} = \frac{(\n - 3)!}{(l - 1)!(\n - l - 2)!} \mid (\n - 3)!$,
and $\binom{\n - 3}{l - 2} = \frac{(\n - 3)!}{(l - 2)!(\n - l - 1)!} \mid (\n - 3)!$,
so we may take $k' = (\n - 2)$.

If $l \leq \frac{\n}{2}$, then we take $m' = 2(l - 1)!(\n - l - 2)!$, and
\begin{multline*}
  m' \abs{D} =
  2 (l - 1)! (\n - l - 2)! \left(
  (\n - 2l) \binom{\n - 1}{l} + (\n - 2) \binom{\n - 2}{l - 1} \right) \\
  = 2 (l - 1)! (\n - l - 2)! \left(
  \frac{(\n - 2l)(\n - 1)!}{l!(\n - l - 1)!} + \frac{(\n - 2)(\n - 2)!}{(l - 1)!(\n - l - 1)!} \right) \\
  = 2 (\n - 2)! \left(
  \frac{(\n - 2l) (\n - 1)}{l (\n - l - 1)} + \frac{\n - 2}{\n - l - 1} \right).
\end{multline*}
The last multiplicand is
\[
\frac{\n(\n - l - 1)}{l(\n - l - 1)} = \frac{\n}{l} \leq \n.
\]

If $l > \frac{\n}{2}$, then we take $m' = 2(l - 2)!(\n - l - 1)!$, and
\begin{multline*}
  m' \abs{D} =
  2 (l - 2)! (\n - l - 1)! \left(
  (2l - \n) \binom{\n - 1}{l - 1} + (\n - 2) \binom{\n - 2}{l - 1} \right) \\
  = 2 (l - 2)! (\n - l - 1)! \left(
  \frac{(2l - \n)(\n - 1)!}{(l - 1)!(\n - l)!} + \frac{(\n - 2)(\n - 2)!}{(l - 1)!(\n - l - 1)!} \right) \\
  = 2 (\n - 2)! \left(
  \frac{(2l - \n) (\n - 1)}{(l - 1)(\n - l)} + \frac{\n - 2}{l - 1} \right).
\end{multline*}
The last multiplicand is
\[
\frac{\n(l - 1)}{(l - 1)(\n - l)} = \frac{\n}{\n - l} \leq \n.
\]

So the total number of voters is
\begin{multline*}
  \binom{\n}{2} \left((\n - 2) \abs{T_z} k' + m' \abs{D} \right) \\
  \leq \frac{\n(\n - 1)}{2} \left( (\n - 2) 3 (\n - 3)! (\n - 2) + 2 (\n - 2)! \n \right) \\
  = \n (\n - 1) (\n - 2)! \left( \frac{3}{2} (\n - 2) + \n \right) \\
  = \n! \left( \frac{3}{2} \n - 3 + \n \right) < \frac{5}{2} \n \, \n!.
\end{multline*}

\subsection*{Chaotic}
Given an arbitrary function, we wish to create a ``single-edged'' $d_{x,y}$, as
developed in Section~\ref{elsesec}, for each edge in the function; there are at
most $\binom{\n}{2}$ of these.  In creating such a single-edged $d_{x,y}$, we
use a collection of symmetric functions, under permutations which take $x$ and
$y$ to elements having certain valence combinations.  Reading the proof of
Claim~\ref{imbalclm} carefully, we see that at most 4 such valence pairs are
used.  For each of these, all permutations fixing the relevant
pair are used, giving $4 (\n - 2)!$ choice functions used in
the creation of the single-edged function.  However, these are being combined
by some set of rational coefficients, so in fact we must expand each such
function taking as many copies as the numerator of its associated coefficient
when put over the least common denominator $L$.  The coefficients sum to 1, so
the sum of the numbers of copies is $L$; so $L (\n - 2)!$ functions are used for
each edge.

Six possible linear combinations of valence pairs yielding $\ntnt$ were
considered in Claim~\ref{imbalclm}.  Using the fact that the original valence
pairs are pairs of integers between $-\n$ and $\n$, we can solve for the
coefficients by matrix inversion and determine that the least common denominator
is at most the determinant of the associated matrix.
For instance, in Case~\ref{threecase}, suppose there is a point of $y = -x$
strictly between two points of $V_1(c)$ on a line segment not parallel to $y =
x$.  (This is true in this Case, unless $V_1(c)$ is contained in a line parallel
to $y = x$ except for a point on the line $y = -x$; such a situation actually has
a smaller common denominator.)

We use Claim~\ref{parlineclm}, with $\ntnt$ for $\bar{w}$, and note that either
the line from $\bar{p}_0$ through $\bar{w}$ meets the line $L$, or the
line from $\bar{p}_1$ through $\bar{w}$ does.  Say $\bar{p}$ is the one which
works.  Then $\ntnt$ is an imbalanced linear combination of the three valence pairs
$\bar{p}$, $\bar{q}_0$, and $\bar{q}_1$.  Say $\bar{p} = (a,b)$; then one of the
$\bar{q}$ is $(b, a)$ and the other is $(c,d)$.  Solving
\begin{align*}
r + s + t = 1\\
a r + b s + c t = 0\\
b r + a s + d t = 0
\end{align*}
amounts to
\[
\left[ {\begin{array}{c}
r\\
s\\
t
\end{array}} \right]
= \left[ \begin{array}{ccc}
1 & 1 & 1 \\
a & b & c \\
b & a & d
\end{array} \right]^{-1}
\left[ \begin{array}{c}
1 \\
0 \\
0
\end{array} \right].
\]
Since everything on the right hand side is an integer, the only denominator
introduced is the determinant of the 3 by 3 matrix,
$bd - ac - ad + bc + a^2 - b^2$.  Since each of $a, b, c, d$ is an integer
between $-\n$ and $\n$, this is at most $6 \n^2$.

In the cases involving four points, such as Case~\ref{fourcase1} or
Case~\ref{fourcase2}, an additional constraint is needed, which is not a consequence of
the other three.  In Case~\ref{fourcase1}, use $r_1 + r_2 + r_3 + r_4 = 1$,
$a_1 r_1 + a_2 r_2 + a_3 r_3 + a_4 r_4 = 0$,
$b_1 r_1 + b_2 r_2 + b_3 r_3 + b_4 r_4 = 0$, and add
$b_3 r_3 + b_4 r_4 = -a_3 r_3 - a_4 r_4$, since the points $\bar{v}$ and $(k_0,
-k_0)$ were linear combinations of two valence pairs on the line $y = -x$.
The lowest common denominator is at most the determinant of
\[
\left[ \begin{array}{cccc}
1 & 1 & 1 & 1 \\
a_1 & a_2 & a_3 & a_4 \\
b_1 & b_2 & b_3 & b_4 \\
0 & 0 & a_3 + b_3 & a_4 + b_4
\end{array} \right],
\]
which expands to a 16-term sum of three $a_i$ or $b_i$ each, hence at most $16
\n^3$.
In Case~\ref{fourcase2}, note that the point $\bar{v}$ between two points of
$V_1(c)$ is arbitrary; we can choose any point between the two.  Between any
two pairs of integers on which do not lie on a line perpendicular to $y = x$,
we can choose a point whose projection on $y = x$ is of the form
$\left(\frac{l}{4}, \frac{l}{4}\right)$ for some integer $l$.  Solving
\[
\left[ \begin{array}{cccc}
1 & 1 & 1 & 1\\
k & a & b & c\\
k & b & a & d\\
0 & a & b & c
\end{array} \right]^{-1}
\left[ \begin{array}{c}
1 \\
0 \\
0 \\
\frac{l}{4}
\end{array} \right]
\]
gives a determinant $k (2a^2 - b^2 - ab - ac + bd + bc - ad)$ at most $8 \n^3$;
we must multiply by 4 because of the fraction in the column vector, giving $32
\n^3$.  This is the largest denominator of the six possible linear combinations
in Claim~\ref{imbalclm}.

Thus our upper bound is
\[
L (\n - 2)! \frac{\n^2 - \n}{2} < 16\, \n^3 \n (\n - 1) (\n - 2)! = 16\, \n^3 \n!.
\]

\newpage
\bibliography{listx,816}

\begin{thebibliography}{6}
\providecommand{\natexlab}[1]{#1}
\providecommand{\url}[1]{\texttt{#1}}
\expandafter\ifx\csname urlstyle\endcsname\relax
  \providecommand{\doi}[1]{doi: #1}\else
  \providecommand{\doi}{doi: \begingroup \urlstyle{rm}\Url}\fi

\bibitem[Alon(2002)]{Alo02}
N.~Alon.
\newblock {Voting paradoxes and digraphs realizations}.
\newblock \emph{Advances in Appl. Math.}, 29:\penalty0 126--135, 2002.

\bibitem[Condorcet(1785)]{condorcet}
J.~A. N.~d. Condorcet.
\newblock \emph{Essai sur l'application de l'analyse \`a la probabilit\'e des
  d\'ecisions rendues \`a la pluralit\'e des voix}.
\newblock L'imprimerie royale, Paris, 1785.

\bibitem[Erd\H{o}s and Moser(1964)]{ErMo64}
P.~Erd\H{o}s and L.~Moser.
\newblock {On the representation of directed graphs as unions of orderings}.
\newblock \emph{Magyar Tud. Akad. Mat. Kutats Int. Kvzl.}, 9:\penalty0
  125--132, 1964.

\bibitem[McGarvey(1953)]{McG53}
D.~C. McGarvey.
\newblock {A theorem on the construction of voting paradoxes}.
\newblock \emph{Econometrica}, 21:\penalty0 608--610, 1953.

\bibitem[Shelah(2009)]{shelah2009mdp}
S.~Shelah.
\newblock What majority decisions are possible.
\newblock \emph{Discrete Mathematics}, 309\penalty0 (8):\penalty0 2349--2364,
  2009.

\bibitem[Stearns(1959)]{Ste59}
R.~Stearns.
\newblock {The voting problem}.
\newblock \emph{American Mathematical Monthly}, 66:\penalty0 761--763, 1959.

\end{thebibliography}
\bibliographystyle{abbrvnat}

\end{document}